\documentclass[10pt,reqno]{amsart}
\usepackage{tikz-cd}
\usepackage[new]{old-arrows}
\usepackage{amsmath,amsthm,amscd,amsfonts,amssymb,euscript,latexsym,mathrsfs,tikz}
\usepackage{mathdots}
\usepackage{xcolor}
\usepackage{hyperref}
\hypersetup{colorlinks=true,linkcolor=blue, citecolor=blue}
\usepackage{frcursive}

\usepackage[matrix,arrow]{xy}
\input{amssym.def}
\input{xy}

\xyoption{all}
\theoremstyle{remark}

\theoremstyle{plain}
\newtheorem*{theorem*}{Theorem}
\numberwithin{equation}{subsection}

\newtheorem{guess}{Theorem}[section]

\newtheorem{thm}[guess]{Theorem}
\newtheorem{lem}[guess]{Lemma}
\newtheorem{prop}[guess]{Proposition}
\newtheorem{Cor}[guess]{Corollary}
\newtheorem{defi}[guess]{Definition}

\theoremstyle{definition}
\newtheorem{rem}[guess]{Remark}

\newcommand{\Hfr}{\mathfrak{H}}

\newcommand{\gfr}{\mathfrak{g}}

\newcommand{\cV}{\mathcal{V}}

\newcommand{\cQ}{\mathcal{Q}}
\newcommand{\cO}{\mathcal{O}}

\newcommand{\cE}{\mathcal{E}}
\newcommand{\cG}{\mathcal{G}}

\newcommand{\cB}{\mathcal{B}}
\newcommand{\cF}{\mathcal{F}}
\newcommand{\cH}{\mathcal{H}}
\newcommand{\cM}{\mathcal{M}}
\newcommand{\cA}{\mathcal{A}}

\newcommand{\cR}{\mathcal{R}}
\newcommand{\cL}{\mathcal{L}}
\newcommand{\cT}{\mathcal{T}}
\newcommand{\cS}{\mathcal{S}}

\newcommand{\mf}{\mathfrak{f}}

\newcommand{\Pic}{\mathrm{Pic}}

\newcommand{\Mor}{\mathrm{Mor}}

\newcommand{\Spec}{\mathrm{Spec}}

\newcommand{\glue}{\mathrm{\tt glue}}
\newcommand{\lra}{\longrightarrow}

\newcommand{\hra}{\hookrightarrow}

\newcommand{\ra}{\rightarrow}
\newcommand{\ol}{\overline}

\newcommand{\spec}{{\rm Spec}\,}

\newcommand{\PP}{\mathbb{P}}

\newcommand{\ZZ}{\mathbb{Z}}
\newcommand{\GG}{\mathbb{G}}

\newcommand{\XX}{\mathbb{X}}
\newcommand{\CC}{\mathbb{C}}

\newcommand{\tcm}{\text{\cursive c}}
\newcommand{\tcs}{\small\text{\cursive s}}
\newcommand{\tcg}{\small\text{\cursive g}}
\newcommand{\tcb}{\it b}
\newcommand{\tcge}{\small\text{\cursive ge}}
\newcommand{\Id}{\mathrm{Id}}

\newcommand{\Aut}{\mathrm{Aut}}
\begin{document}

\title{Conformal blocks, parahoric torsors and Borel-Weil-Bott}

\author[V. Balaji]{Vikraman Balaji}
\address{Chennai Mathematical Institute }
\email{balaji@cmi.ac.in}

\author[Y. Pandey]{Yashonidhi Pandey}
\thanks{The research for this article was partially funded by the SERB Core Research Grant CRG/2022/000051.}
\address{ 
Indian Institute of Science Education and Research, Mohali Knowledge city, Sector 81, SAS Nagar, Manauli PO 140306, India}
\email{ ypandey@iisermohali.ac.in, yashonidhipandey@yahoo.co.uk}

\begin{abstract} Let $X$ be a smooth projective curve over a field $k$. Let $\mathcal{G}$ be a parahoric group scheme on $X$ as  in \cite{pr}. The aim of this article is to establish the basic themes of the theory of conformal blocks in the parahoric setting. Using Hecke correspondences, we reduce questions on parahoric stacks  to those on the reductive (or hyperspecial parahoric) stack via the Iwahori stack. Following this {\em reduction approach}, we set-up relationships between the cohomology of line bundles on various moduli stacks of torsors. These relationships give a proof of \cite[Conjecture 3.7]{pr}  in characteristic zero, the principle of propagation of vacua and a direct proof of the independence of central charge on base points. Projective flatness is recovered as a corollary of Faltings' construction of the Hitchin connection.  Using Teleman's basic results (\cite{bwb}), we deduce the analogous result that cohomology of line bundles on the stack of principal $G$-bundles vanishes in all degrees except possibly one over $\mathbb{C}$. Results on twisted vacua \cite{hongkumar} are obtained as immediate consequences.

\end{abstract}
\maketitle

\small
\tableofcontents
\setcounter{tocdepth}{2} % Set depth to 1
\normalsize

%\Large

\section{Introduction}

This article is in the nature of a series of remarks. It arose out of an attempt to understand the classical theory of conformal blocks and cohomology of bundles in the setting of moduli stacks of torsors for parahoric Bruhat--Tits group schemes over a smooth projective curve. Classical Bruhat--Tits group schemes are defined on discretely valued Hensel rings.  In \cite{pr}, a global version $\cG$ of such group schemes was  introduced over a smooth projective curve $X$.  Let us state the framework of this article.

{\em We assume that the base field  $k$ is of arbitrary characteristics up to \S \ref{Picard}, and, of characteristic $0$ in \S \ref{ConformalBlocksection} and any section depending on it.} 

Let $G$ be an {\em almost simple,  simply-connected} connected group scheme over $k$. Let $X$ be a smooth projective curve over $k$ with function field $k(X)$ and let $\cG$ be a smooth, affine,  connected group scheme over $X$ such that its generic fibre $\cG_{k(X)}$ is {\em absolutely simple, semisimple, simply connected and splits over a tamely ramified extension}. We call such group schemes {\it tamely split}. We wished to provide  transparent strategies towards the computation of the cohomology of the space of sections of all line bundles and evaluation vector bundles on the moduli stack  $\cM(\cG)$ of  torsors on the curve $X$ under a Bruhat--Tits group scheme $\cG$. 

 By piecing together classically known facts under the new perspective of Hecke correspondences, we prove a conjecture of Pappas-Rapoport \cite[3.7]{pr}  in Theorem \ref{papparappa1}. The proof is essentially an observation. The propagation of vacua and subsequently the Verlinde isomorphism for the space of global sections for line bundles on $\cM_X(\cG)$  mentioned in \cite[8.2]{pr2024} are deduced in Theorem  \ref{propofvacua} and Corollary \ref{thmE}.

The classical theory of conformal blocks was enunciated over three decades ago by Tsuchiya-Ueno-Yamada \cite{tuy}, Beilinson-Schechtman \cite{besc}, Faltings \cite{faltings}, Beauville \cite{beau}, Laszlo, Sorger \cite{ls} and several others.  Our goal is to take up some basic features of these works and view them in the setting of parahoric torsors. In this spirit, we also approached the basic articles of Teleman \cite{telemanfusion},     \cite{telemanverlinde} and \cite{bwb}, on higher cohomologies of vector bundles on moduli stacks. We primarily rely on the articles cited above. Building on the basic computations in \cite{faltings} and \cite{bwb}, we go on to make some remarks on {\tt BWB} type results of Teleman \cite{bwb}.

  The landscape of parahoric torsors could be viewed in one single large frame. Our key tool  for this is the setting of Hecke correspondences  defined in \cite[8.2.1]{bs}  (see \cite[\S 7.5, 7.6]{me} for details as well as \eqref{IwahoritoGstack} and \eqref{Heckemodification}). We make  a finer examination of the Hecke-correspondence  basing ourselves on the foundational work of Bruhat--Tits \cite{bt}. The term Hecke correspondences in the setting of moduli goes back to the early work of Narasimhan and Ramanan.
 
 The central theme in the article may be called a {\em reduction approach}, i.e., using the Hecke motif to reduce questions on parahoric stacks  to those on the reductive (or hyperspecial parahoric) stack via the Iwahori stack. This reduction couples well with the uniformization theorem whereby, one has the flexibility to play with the base points for the atlases (denoted $A \subset X$) as well as the ramification points where the parahoric structures are prescribed (denoted $\cR \subset X$). These correspondences can be flipped around so that one can answer questions even on the reductive stack via the parahoric ones  (see \cite{bpnr}). An illustration of this is in the proof of the propagation of vacua which makes it transparent.
 
 We note that this mobility across spaces  via correspondences is not quite possible in the equivariant setting.  The equivariant moduli space on a curve $Y$ when interpreted as a parahoric stack on the curve $Y/\Gamma$ now forgets the group $\Gamma$ and has the flexibility to utilise varying ramified covers $Y$ keeping the parahoric data constant, especially because this data is completely {\em local}. As an illustration of this phenomenon, we refer the reader to  \S \ref{elaboration}.
 
 As applications, we prove the following: 
\begin{itemize}
\item X. Zhu has shown (see \cite[Proposition 4.1, and page 8]{zhu})  that the central charge morphism (see \S \ref{centralchargestack}) on the Picard group of $\cM(\cG)$  is independent of the base point chosen to define it. {\em In Corollary \ref{ccconstantcy}, we give a new proof of this result}. The characteristic of $k$ is arbitrary here.

\item The article by Hong and Kumar \cite{hongkumar} (cf. also \cite{swarnavatanmay} for related material)  and the more recent one \cite{hongkumar2}, takes up the study of twisted conformal blocks. The objects of study involve a Galois cover $\psi:Y \to X$ with Galois group $\Gamma$ and simple Lie algebras $\gfr$ equipped with a homomorphism $\rho:\Gamma \to \text{Aut}(\gfr)$. { Barring a few minor exceptions involving the tamely split cases, which we have elaborated, the results in {\em loc.cit} follow as corollaries of our main results.} 
\item We address the question of higher cohomology of line bundles on $\cM(\cG)$ in the spirit of Teleman's {\tt BWB} theorems. {\em In the split case, this also answers a  query in \cite[page 169]{pr}}. We get interesting combinatorial integers $\tcb_{_{\pi}}$ \eqref{b1} and $\tcb_{_h}$ \eqref{b2} relating the Hecke transforms of the cohomologies of bundles on the parahoric stack $\cM(\cG)$ with those  on $\cM(\Hfr)$ the stack of torsors for the semisimple group scheme $\Hfr$.
\item The article by Biswas-Mukhopadhyay-Wentworth \cite{biswas-mukho-wentworth} shows the existence of a canonical Hitchin connection on the moduli stack of parabolic $G$-bundles for a family of curves $C \to S$.  This along with the parahoric cases, is essentially deduced as an immediate corollary of the classical theorem of Faltings \cite[Theorem V.2, page 563]{faltings1993}. This also gives as a corollary, the result \cite[Theorem C]{hongkumar} (see \S \ref{recovery}).
\end{itemize}

After we finished this work, we came across the article \cite{dh} by Damiolini and Hong. There are natural similarities in the proof of Theorem \ref{papparappa1}, but the results in our article are quite different and the setting somewhat more general (for details on this, see \S \ref{clarifications}, \S\ref{clarifications1}, Corollary \ref{moregen} and Proposition \ref{redtoIwa}).

{\it Acknowledgments}: {The second author thanks Norbert Hoffmann for discussions in 2018 and Mary Immaculate College for hospitality. We sincerely thank an anonymous referee for the generous suggestions and comments on the earlier versions of this article. We believe that the comments have led to a better exposition in many ways.}

\section{Loop groups and Bruhat--Tits group schemes} \label{lgpthedata}

\subsection{Notations and Conventions} Our base field will always be $k$. Let $X$ be a smooth projective curve over $k$. For $x \in X$, let $\cO_x$ denote the {\it complete} local ring at $x$ (i.e. $\cO_x \simeq k[\![t]\!]$)  and ${K}_x \simeq k(\!(t)\!)$ be its field of fractions, the corresponding local field.   
The set $\cR \subset X$ will denote the set of ramification points where the parahoric structure gets prescribed. For $x \in X$, let $\mathbb{D}_x :=\Spec({\cO}_x)$ and $\mathbb{D}^{^\times}_x :=\Spec(K_x)$.

For our purposes, $\cG$ will be a smooth affine connected $X$-group scheme such that $\cG_{k(X)}$ is an {\em absolutely simple, semisimple, simply-connected} ${k(X)}$-group which {\em splits over a tamely ramified extension of ${k(X)}$}.  Classical Steinberg theory shows also that the group $\cG_{k(X)}$ is {\em quasi-split}.   Since for us $k$ is algebraically closed, at the local fields $K_x$, the groups $\cG_{K_x}$ are again automatically quasi-split by Steinberg's theorem.
  
 We shall fix a maximal torus $T$, its apartment $\cA_T$ and an origin $v_0$ of $\cA_T$. At a closed point $x$ of $X$, we have parahoric subgroups $\cG(\cO_x) \subset \cG({K}_x)$ in the sense of Bruhat--Tits (see \cite{bt}).   Let ${\tt F}_x$ be a facet of $\cA_T$ determined by this inclusion. We choose a ${K_x}$-Borel subgroup $B$ in the semisimple group $\cG_{{K_x}}$ containing $T$. Let ${\bf a}_0$ denote the unique alcove in $\cA$ whose closure  contains $v_0$ and is contained in the finite Weyl chamber determined by our chosen Borel subgroup. This determines a set $\mathbb{S}$ of simple {\it affine roots} $\alpha$.  Let $\theta$ be the highest root of $\cG_{K_x}$ and $\theta^{\vee}$ denote the highest coroot of $\cG_{K_x}$. Following Tits \cite{titsjoa}, we denote by $\alpha_{_0}:=1 - \theta$. We will identify the set $\mathbb{S} - \alpha_{_0}$  with the set of simple roots of $\cG_{K_x}$ relative to $B$.

For a root $\alpha$ of $\cG_{K_x}$, let $\alpha^{\vee}$ denote the corresponding coroot. For simple roots $\alpha \in \mathbb{S} - \alpha_{_0}$ of $\cG_{K_x}$, let $a_{\alpha^\vee}$ be integers defined by the relation  \begin{equation} \label{corootcoeff}  \theta^\vee=\sum_{\alpha \in \mathbb{S} - \alpha_{_0}} a_{\alpha^\vee} \alpha^\vee.  \end{equation} 
By \cite[\S 6.1]{kac}, $a_{\alpha_{_0}}^{\vee}=1$ in all cases.

\subsection{Some attributes of facets}\label{attributesoffacets}
Let ${\tt F} \subset \cA$ be a facet. We denote by $\cG_{{\tt F}}$ the parahoric group scheme over a discrete valuation ring $\cO$ with fraction field $K$ (see \cite[Section 1.7]{bt}).  Recall the loop group and the parahoric loop group representing the following functors:
\begin{equation} 
R \mapsto L\cG_{{\tt F}}(R) := \cG_{{\tt F}}(R \otimes_k K)~~ (= \cG_{{\tt F}}(R(\!(t)\!)))
\end{equation}
\begin{equation}
R \mapsto L^+\cG_{{\tt F}}(R) := \cG_{{\tt F}}(R \otimes_k \cO)~~ (= \cG_{{\tt F}}(R[\![ t ]\!]))
\end{equation}

The flag variety associated to a facet ${\tt F}$ is the quotient   
\begin{equation} \label{flagv}
\cF\ell_{{\tt F}} := L\cG_{{\tt F}}/L^+\cG_{{\tt F}}
\end{equation}
of fpqc-sheaves. It is represented by an ind-projective scheme over $k$ .

We denote the subset of simple affine roots  not vanishing on ${\tt F}$ by  \begin{equation}\label{sx}
S({\tt F}) \subset \mathbb{S}.
\end{equation}
Let $\Omega_{\tt F}=\{ \omega_\alpha | \alpha \in S({\tt F})\}$ denote  the fundamental weights of $G_K$ corresponding to $S({\tt F})$ where for $\alpha_{_0}$ we take the trivial weight.  For any facet ${\tt F}$ we define 
\begin{equation} \label{localofF} l({\tt F})= {\tt gcd} \{ a_{\alpha^\vee} | \alpha \in S({\tt F})  \}.
\end{equation} Let $\cR \subset X$ be a finite subset. For a set $\{{\tt F}_x \}_{x \in \cR}$ of facets, we define \begin{equation} \label{f}
 \ell := {\tt lcm} \{ l ({\tt F}_x) | x \in \cR \}.
 \end{equation}

\subsection{From Bruhat--Tits \cite{bt}}\label{frombt} We recall some foundational material from {\em loc.cit} pages 134-138. The perspective we wish to give on these results, especially in relationship to the theme of the present article make these  central to the constructions in this article.

Let $\cA(\tt F)$ denote the {\em star of $\tt F$} in the simplicial complex $\cA$, i.e., the set of facets $\tt F'$ of $\cA$ such that $\tt F \subset \bar{\tt F}'$ and ordered by the relation $\tt F' \leq \tt F''$ if $\tt F' \subset \bar{\tt F}''$.

Let $\cG_{\tt F}$, and $\cG_{\tt F'}$ be the associated Bruhat--Tits group schemes (which we assume to be connected) on $\spec \cO$ with generic fibre $\cG_K$. Then, the identity map on $\cG_K$ extends to a morphism:
\begin{equation}
i_{_{\tt F,\tt F'}}:\cG_{\tt F'} \to \cG_{\tt F}.
\end{equation}
Let $\bar{i}_{_{_{{\tt F},{\tt F}'}}}$ denote the induced morphism:
\begin{equation}
\bar{i}_{_{_{{\tt F},{\tt F}'}}}:\cG_{_{{\tt F'},k}}  \to \cG_{_{{\tt F},k}}
\end{equation}
over the residue field $k$ (which for us is algebraically closed). By \cite[4.6.33]{bt}, the map:
\begin{equation}
F' \mapsto \text{Im}(\bar{i}_{_{_{{\tt F},{\tt F}'}}})
\end{equation}
is an isomorphism of the ordered subsets of the star $\cA({\tt F})$ of ${\tt F}$ onto the set of parabolic subgroups of the closed fibre $\cG_{_{_{{\tt F},k}}}$. For an independent self-contained proof of this fact in the generically split case, see \cite[Corollary 4.3.12]{pp}.

Furthermore, the inverse image of $\text{Im}(\bar{i}_{_{_{{\tt F},{\tt F}'}}})$ in $\cG_{\tt F}(\cO)$ is equal to $\cG_{\tt F'}(\cO)$.  Let 
\begin{equation}
P_{_{_{{\tt F},{\tt F}'}}}:= \text{Im}(\bar{i}_{_{_{{\tt F},{\tt F}'}}})
\end{equation} 
be the parabolic subgroup of $\cG_{_{{\tt F},k}}$ determined by ${\tt F}' \in \cA(F)$. In other words, the subspace of sections in $\cG_{\tt F}(\cO)$ whose canonical image under the residue map lies in the parabolic $P_{_{_{{\tt F},{\tt F}'}}}$ is precisely $\cG_{{\tt F}'}(\cO)$. The last statement implies the following results which are immediate and are basic for our perspective.
\begin{enumerate}
\item We consider the loop analogues of the parahoric subgroups and we see immediately that for any pair $(\tt F, \tt F')$ such that $\tt F \subset \bar{\tt F}'$, we have an inclusion 
\begin{equation}\label{BTconseqs}
L^+\cG_{{\tt F'}} \subset L^+\cG_{{\tt F}}
\end{equation}
and the fpqc quotient $L^+\cG_{{\tt F}}/L^+\cG_{{\tt F'}}$ is a smooth projective homogeneous $k$-variety isomorphic to ${\cG_{_{{\tt F},k}}\over P_{_{_{{\tt F},{\tt F}'}}}}$.

\item We have the group scheme isomorphism  
\begin{equation}\label{nerondilat}
\cG_{{\tt F}'} \simeq \cG^{^{P_{_{_{{\tt F},{\tt F}'}}}}}_{\tt F}  
\end{equation} of $\cG_{{\tt F}'}$ with the {\em Neron blow-up}, or dilatation of the group scheme $\cG_{\tt F}$ along the subgroup $P_{_{_{{\tt F},{\tt F}'}}}\subset \cG_{_{{\tt F},k}}$.

\item Let $\tt F_{_j} \subset \bar{\tt F'}$, $j = 1,2$, i.e., the facet $\tt F'$ lies in the star $\cA(\tt F_{_j})$ for $j =1,2$. Then we have a diagram of morphisms of affine group schemes over $\spec \cO$:

\begin{equation}\label{heckeatgplevel}
\begin{tikzcd}
	& {\cG_{\tt F'}} \\
	{\cG_{_{\tt F_1}}} && {\cG_{_{\tt F_2}}}
	\arrow["{i_{_{\tt F_1,\tt F'}}}"', from=1-2, to=2-1]
	\arrow["{i_{_{\tt F_2,\tt F'}}}", from=1-2, to=2-3]
\end{tikzcd}\end{equation} 

\item We now make a few key observations regarding the affine flag varieties $\cF \ell_{\tt F}$ \eqref{flagv}. Given a pair $(\tt F, \tt F')$ such that $\tt F \subset \bar{\tt F}'$, we get the canonically induced morphism:
\begin{equation}\label{theflagfibres}
\mf_{_{_{{\tt F},{\tt F}'}}}:\cF\ell_{\tt F'} \to \cF\ell_{\tt F}
\end{equation}
which is an \'etale locally trivial fibration with fibre type isomorphic to the homogeneous space ${\cG_{_{{\tt F},k}}\over P_{_{_{{\tt F},{\tt F}'}}}}$.

\item Let $\mathbf a$ correspond to the open facet giving rise to the Iwahori group scheme $\cG_{\mathbf a}$. For any facet $\tt F$, we have a morphism $i_{_{\tt F,\mathbf a}}:\cG_{\mathbf a} \to \cG_{\tt F}$. In this situation, the parabolic subgroup $P_{_{_{{\tt F},{\mathbf a}}}}$ is in fact a Borel subgroup of the closed fibre ${\cG_{_{{\tt F},k}}}$ (obtained by the inverse image of a Borel of the Levi quotient) and the fibre of $\mf_{_{_{{\tt F},{\mathbf a}}}}$ is isomorphic to 
\begin{equation}\label{fromiwahori}
{\cG_{_{{\tt F},k}}}\over B_{_{{\tt F},k}}.
\end{equation}

\end{enumerate}
\section{The moduli stacks $\cM(\cG)$ and their uniformization} \label{modulistack}
\subsection{The group scheme} \label{gpsch} For $x \in X$, let $\mathbb{D}_x :=\Spec({\cO}_x)$ and $\mathbb{D}^{^\times}_x :=\Spec(K_x)$. Let $\cR \subset X$ be a non-empty finite set of closed points. For each $x \in \cR$, we choose a facet ${\tt F}_x \subset \cA_T$. As in the introduction, let $\cG$ be a group scheme on $X$ which is semisimple on $X - \cR$, splits after going to a tamely ramified extension $L$ of the function field of $X$, with absolutely simple fibers and restricts to $\mathbb{D}_x$ as the parahoric group scheme $\cG_{{\tt F}_x} \ra \mathbb{D}_x$ corresponding to ${\tt F}_x$.  Throughout this article, such a group scheme will be called {\bf tamely split}.

A {\it quasi-parahoric} torsor $\cE$ is simply a $\cG$--torsor on $X$ for a parahoric group scheme $\cG$.  Let $\cM(\cG)$ denote the moduli stack of quasi-parahoric $\cG$-torsors on $X$. This is an algebraic stack (see \cite[Proposition 4.1]{heinloth}). 

At each marking $x \in \cR$, we choose an alcove $\mathbf{a}_x $ in whose closure the facet ${\tt F}_x$ lies. Let $\cG^{\tt I}$ be the group scheme (of the Iwahori type) obtained from $\cG$ by gluing with  $\mathcal{G}_{\mathbf{a}_x}$ instead of $\mathcal{G}_{{\tt F}_x}$ (see \cite[\S 7.5]{me} or \cite[\S 6.1]{pp}). The morphism $\cG^{\tt I} \ra \cG$ of group schemes induces by an extension of structure group, the following morphism of algebraic stacks: 
\begin{equation} \label{IwahoritoGstack}
\pi: \cM(\cG^{\tt I}) \ra \cM(\cG)
\end{equation} 
We  will sometimes abbreviate $\cM:= \cM(\cG)$ and set $ \cM^{\tt I}:=\cM(\cG^{\tt I})$.

\subsection{Uniformization of $\cM(\cG)$} \label{uniformization}
 Let $A \subset X$ be a finite non-empty set of closed {\bf atlas} points. These are also called  {\it punctures}  in Teleman \cite{bwb}. For each $x \in A$, let ${\tt F}_x \subset \cA$ denote the facet determined by $\cG$ in a formal neighbourhood around $x$. For each $x \in A$, let us choose a uniformizer $t_x \in {\cO}_{x}$. Then we have the map
\begin{equation}\label{glue} 
\tcg:= \glue_{\{t_x| x \in A \}}: \prod_{x \in A} \cF\ell_{{\tt F}_x} \ra \cM(\cG),
\end{equation}
which on the coset $(\cdots, f_x L^+ \cG_{{\tt F}_x}, \cdots)_{\{x \in A \}}$ glues the trivial $\cG$-torsor on $X - A$ with the trivial $\cG_{{\tt F}_x}$-torsor on $\Spec({\cO}_{x})$ via $f_x$. It will be convenient to abbreviate this morphism as $\tcg$. By \cite[Theorem 4]{heinloth}, this morphism has \'etale local sections. It follows that the morphism ${\pi_{_{\tcs}}}: \cM^{\tt I} \ra\cM(\cG)$ is representable.

An $R$-point of $\prod_{x \in A} \cF\ell_{{\tt F}_x}$ classifies $\cG$-torsors on $X \times \Spec(R)$ together with a section on $(X - A) \times \Spec(R)$. The map $\prod_{x \in A} \cF\ell_{{\tt F}_x}  \ra \cM$ forgets the section and the ind-scheme $\cG(X - R)$ acts on $\prod_{x \in A} \cF\ell_{{\tt F}_x}$ by changing the section.  By the uniformization theorem for $\cG$-torsors (cf \cite{heinloth}) we have an isomorphism of stacks
\begin{eqnarray} \label{unif1}
\tcg: \prod_{x \in A} \cF\ell_{{\tt F}_x} \ra \cM(\cG) \simeq \cG(X - A) \backslash \prod_{x \in A} \cF\ell_{{\tt F}_x}.
\end{eqnarray}
{\em We emphasize that we allow for the atlas point set $A$ to be independent of the markings $\cR$ for flexibility of application}.  In particular, we allow for $A =\cR$ as in \cite{me} for Proposition \ref{redtoIwa} or for a $y \in X - \cR$, $A =\{y \}$ as in \cite{ls} for Theorem \ref{propofvacua}.

\section{The Hecke Correspondence} \label{HeckeModdiag}

Let $\cG$ be a parahoric Bruhat--Tits group scheme whose restriction to $X - \cR$ is semisimple. Let $\tcs = \{\tcs_x\}_{x \in \cR}$ denote the data of gluing functions for the group scheme $\cG$, and $x \in \cR$, let ${\tt F}_x \subset \cA$ denote the facet determined by $\cG$ in a formal neighbourhood around $x$. Recall that for each $x \in \cR$ we have chosen an alcove $\mathbf{a}_x $ in whose closure the facet ${\tt F}_x$ lies. We denote by $v_{0,x}$, the unique vertex of $\mathbf{a}_x$ such that $\alpha_{_0}$ takes a value different from that of the other vertices of $\mathbf{a}_x$. Let $\Hfr$ denote  {\em a semisimple, simply connected group scheme} on $X$ with absolutely simple fibers isomorphic to $\cG$ on $X - \cR$ and isomorphic to $\mathcal{G}_{v_{0,x}}$ on $\mathbb{D}_x$ with gluing functions $\tcs$. Thus,  we have a Hecke correspondence diagram of group schemes by \eqref{heckeatgplevel}:
\begin{equation}\label{heckecorrespgpscheme} \begin{tikzcd}
	& \mathcal{G}^{\tt I} \\
	{\cG} && {\Hfr} 
	\arrow[ from=1-2, to=2-1]
	\arrow[ from=1-2, to=2-3]
\end{tikzcd}\end{equation}

Thus, we may view the stack $\cM(\Hfr)$ also as a stack of  parahoric torsors with {\em hyperspecial structures} at the marked points $\cR$.

To avoid any confusion, we recall that in the literature, one denotes the constant group scheme by $G = G \times_k X$ and by $\cM(G)$ we mean the stack of torsors for the group scheme obtained by the trivial gluing at all markings.

Let $\cM^{\tt I}$ be an Iwahori stack at the markings $\cR$ and gluing data $\tcs$. We  obtain a representable morphism induced by \eqref{heckecorrespgpscheme}  
\begin{equation} \label{Heckemodification} h_{_{\tcs}}: \cM_{_X}(\mathcal{G}^{\tt I}) \rightarrow \cM_{_X}(\Hfr),
\end{equation} 
which will be convenient to denote by the same notation $h_{_{\tcs}}$. Its fibers are $\prod_{x \in \cR} G/B$.    

Again, if we begin with a parahoric stack $\cM_{_X}(\cG)$ where $\cG$ is a parahoric group scheme on $X$ with parahoric structures at $\cR$ and given by  gluing data $\tcs$, then by consistently choosing the same data, {\sf{ associated to the given choice of markings $\cR \subset X$}}, we have the following  Hecke correspondence diagram induced by \eqref{heckecorrespgpscheme}:
\begin{equation}\label{heckecorresp} \begin{tikzcd}
	& {\cM(\mathcal{G}^{\tt I}) = \cM^{\tt I}} \\
	{\cM(\cG)} && {\cM(\Hfr)} 
	\arrow["\pi_{_{\tcs}}", from=1-2, to=2-1]
	\arrow["h_{_{\tcs}}"', from=1-2, to=2-3]
\end{tikzcd}\end{equation}
where $\pi_{_{\tcs}}$ is also an \'etale local fibration having products $\prod_{_{x \in \cR}} {{\cG_{_{{\tt F_{x}},k}}}\over B_{_{{\tt F_{x}},k}}}$ of flag schemes \eqref{fromiwahori} as fibres (see \S \ref{frombt} for details).
\begin{rem} The notation $\pi_{_{\tcs}}$ and $h_{_{\tcs}}$ reflect the fact that the morphism goes towards the ``parahoric stack" or the ``hyperspecial parahoric stack". \end{rem}

\subsection{Construction of compatible atlases $\pi: \cM^{\tt I} \ra \cM(\cG)$ of \eqref{IwahoritoGstack}}  

Note that the trivial $\cG^{\tt I}$ torsor on $X-A$ has Iwahori structures on the points $\cR -A$.
For an arbitrary $A$, by the uniformization theorem \eqref{unif1} a point in $$\prod_{x \in A \cap \cR} \cF\ell_{\mathbf{a}_x} \times \prod_{x \in A - \cR} \cF\ell_{{\tt F}_x}$$ provides a gluing for the trivial $\cG^{\tt I}$ torsor on $X - A$ with $\cG_{\tt{F}_x}$ where $x \in A -\cR$ and with $\cG_{\mathbf{a}_x}$ where $x \in A \cap \cR$. Taking the product map of $\mf_{_{_{{\tt F},{\tt F}'}}}$ of \eqref{theflagfibres} for $x \in A \cap \cR$ to define  $p_A$, the following diagram is commutative :
\begin{equation} \label{fibproddiag}
\xymatrix{
\prod_{x \in A \cap \cR} \cF\ell_{\mathbf{a}_x} \times \prod_{x \in A - \cR} \cF\ell_{{\tt F}_x} \ar[rrr]^{\tcg^{\tt{I}}} \ar[d]^{p_{_{A}}} &&& \cM^{\tt I} \ar[d]^{\pi_{_{\tcs}}} \\
\prod_{x \in A} \cF\ell_{{\tt F}_x} \ar[rrr]^{\tcg} &&& \cM(\cG)
}
\end{equation}
When $A=\cR$,  the diagram becomes cartesian. In the rest of the article, we will often suppress the subscript $A$ from the morphisms.

Taking the product of $\mf_{_{_{{\tt F},{\tt F}'}}}$ of \eqref{theflagfibres} we get:
\begin{equation}\label{heckecorrespflaglevel} \begin{tikzcd}
	& \prod_{x \in A} \cF\ell_{\mathbf{a}_x} \\
	\prod_{x \in A} \cF\ell_{{\tt F}_x} && \prod_{x \in A} \cF\ell_{{v}_{0,x}} 
	\arrow["", from=1-2, to=2-1]
	\arrow[""', from=1-2, to=2-3]
\end{tikzcd}\end{equation}
By the glue morphisms $\tcg$ of \eqref{unif1} which involve the atlas points $A$, the diagram of flag varieties sits compatibly over the diagram \eqref{heckecorresp} which involve the marked points $\cR$.

\section{Picard functor and central charge} \label{Picard}

At a marking $x \in \cR$, we have chosen the alcove $\mathbf{a}_x$ in whose closure the facet ${\tt F}_x$ lies. We now recall some general results for facets which we will apply to ${\tt F}_x$ (resp. $\mathbf{a}_x$). It will be convenient to work without the subscript $x$. Further, for any point $x \in X$, we abbreviate $K_x$ as $K$. Then $(\mathcal{G}_{\tt F_x})_{K_x}$ is independent of $x$ and we denote it simply as $\mathcal{G}_K$. Let $\cF\ell_{\mathbf{a}}$ denote the Iwahori flag variety as in \eqref{flagv}. We recall the following facts which hold for the group schemes $\cG$ and $\Hfr$ which need not be split. Following the notations in \cite{me}, let ${\tt F}^\alpha$ denote the unique vertex of an alcove where $\alpha$ takes a value different from any other vertex.

\begin{prop} \cite[Proposition 10.1]{pradv} \label{PicF} Let $\cG$ be a parahoric group scheme that splits over a tamely ramified extension. There is an isomorphism
$ \Pic(\cF\ell_{\mathbf{a}}) \simeq \ZZ^{\mathbb{S}}$
 defined by the degrees of the restrictions to $\PP^1_\alpha= L^+\cG_{{\tt F}^\alpha}/L^+\cF\ell_{\mathbf{a}} \hra \cF\ell_{\tt I}$. 
\end{prop}
Let $L_{\epsilon_0}$ be the line bundle on $\cF\ell_{\mathbf{a}}$ corresponding to the affine simple root $\alpha_0$ in $\mathbb{S}$. The Mumford-Igusa construction for 
$L_{\epsilon_0}$ gives the basic central extension 
\begin{equation} \label{ucentralextension}
1 \ra \mathbb{G}_m \ra \tilde{L} \mathcal{G}_K \ra L \mathcal{G}_K \ra 1.
\end{equation} 
In loc. cit. $\tilde{L} \mathcal{G}_K$ is called the {\it loop group in Kac-Moody context} and the $\mathbb{G}_m$ above is {\it the central $\GG_m$}.
We denote  the inverse image of the jet subgroup 
$  L^+ \mathcal{G} \hookrightarrow L \mathcal{G}_K$ of the Iwahori group scheme $\cG_{\mathbf{a}}$ by 
\begin{equation} \label{L+} \tilde{L}^+ \mathcal{G}. \end{equation} 
\begin{prop} \label{lb=char} (from the proof of \cite[Proposition 10.1]{pradv}) All line bundles on $\cF\ell_{\mathbf{a}}$ come from characters on $\tilde{L}^+ \mathcal{G}$.
\end{prop}
For a line bundle $L$ on $\cF\ell_{\mathbf{a}}$, let us denote the corresponding character as
\begin{equation} \label{chiL}
\chi_L : \tilde{L}^+ \mathcal{G} \ra \mathbb{G}_m.
\end{equation} 

Let us state the following theorem of Faltings which will be key to this work.

\begin{thm} (see \cite[Theorem 4.2.1]{BHf3} or \cite[Theorem 17, Remark on page 67]{faltings}) \label{faltingsO(1)} Let $X \ra S$ be a curve over an arbitrary base $S$. Let $\Hfr$ be an absolutely simple, semisimple, simply connected  group scheme on $X$. Choose a closed point $x \in X$ and a uniformizer $t \in {\cO}_{x}$. Let the restriction of $\Hfr$ to $\Spec(\cO_x)$ determine a parahoric group scheme corresponding to a point $v$ in the building of $\Hfr_{K_x}$.  Then $\glue^*_{t}: \Pic \left(\cM \left(\Hfr \right) \right) \rightarrow \Pic \left(\cF\ell_{v,x} \right)$ induced by $\glue_{t}: \cF\ell_{v,x} \ra \cM(\Hfr)$, gives an isomorphism of functors inducing  an isomorphism of Picard groups.
\end{thm}
The case $v=v_0$ is the most important. But we allow for $v$ to be a general hyperspecial vertex of the buidling of $\Hfr_{K_x}$ and not necessarily the origin corresponding to the standard parahoric group scheme.

\subsection{The notion of central charge \cite[Remark 10.2 (i)]{pradv}} \label{ccpradv} Let us denote by $L_{\epsilon_\alpha}$ the line bundle on the flag scheme $\cF\ell_{\mathbf{a}}$ which is associated to the affine simple root $\alpha \in \mathbb{S}$. For a facet ${\tt F}$, there is a  {\em central charge homomorphism} 
\begin{equation} \label{ccaffineflagvariety} c: \Pic \left( \cF \ell_{\tt F} \right) \ra \mathbb{Z}
\end{equation} 
which can be defined by pulling back line bundles to $\cF \ell_{\mathbf{a}}$, then by Proposition \ref{lb=char}, it is the degree of the character $\chi_L$ \eqref{chiL} induced on the central $\GG_m \hra \tilde{L}^+ \mathcal{G}$ \eqref{L+}. Let $\tt{A}$ denote the generalised Cartan matrix of affine type of $\mathcal{G}_K$. By \cite[\S 6.1]{kac}, the 
Dynkin diagram $S(^{t}\tt{A})$ of the dual algebra is obtained from $S(\tt{A})$ by reversing the direction of all arrows but keeping the same enumeration of vertices. Let  $$a_{\alpha}^{\vee}$$ denote the labels of the Dynkin diagram of $S(^{t}\tt{A})$. Then $a_{\alpha_{_0}}^{\vee}=1$ in all cases including $A^{(2)}_{2l}$ for which $a_{\alpha_{_0}}=2$. Then for $\alpha \in \mathbb{S}$ we have 
\begin{equation} \label{centralcharge}
c \left(L_{\epsilon_\alpha} \right)= a_{\alpha}^{\vee}.
\end{equation}

\subsection{Central charge of line bundles on moduli stack of $\cG$-torsors} \label{centralchargestack}

In \cite{zhu} a global affine grassmannian $Gr_{\cG} \rightarrow X$ has been constructed. Let $Gr_{\cG,z}$ denote its fibre over $z \in X$. It is naturally isomorphic to  $\cF\ell_{{\tt F}_z}$ for some facet ${\tt F}_z \subset \cA_T$. By \cite[\S  7, 1st paragraph]{heinloth} the central charge of a line bundle on $\cM(\cG)$ 
\begin{equation} \label{ccmodulistack} {\tcm}_z: \text{Pic}(\cM(\cG)) \ra \mathbb{Z}
\end{equation} 
is defined at an arbitrary point $z \in X$ after pulling it back to $Gr_{\cG,z}$ and then applying the central charge homomorphism ${\tcm}_z: \text{Pic}(Gr_{\cG,z}) \to \mathbb Z$ \eqref{ccaffineflagvariety}.  Denoting by $\mathcal{G}_x$ the reduction of $\mathcal{G}$ to the closed fiber at $x \in X$ and by $\XX^*(\cG_x)$ the group of characters on $\cG_x$, by \cite[Theorem 7]{heinloth} we have the exact sequence 
\begin{equation} \label{picseq}
0 \ra \prod_{x \in \cR} \XX^*(\cG_x) \ra \Pic \left(\cM(\cG) \right) \stackrel{{\tcm}_z}{\lra} \ZZ,
\end{equation}
which may not be surjective. 

\begin{rem}\label{faltingsandcc1} ({\sf On the splitness assumption on $\Hfr$})
 Faltings constructs a line bundle 
$L(\cE)_{\Hfr}$ of central charge one on $\cM(\Hfr)$ when $\Hfr$ is a {\em split group scheme}. He remarks \cite[Remark on page 67]{faltings}, that this result generalises when $\Hfr$ is an inner form (or even an outer form) as well. A line by line verification shows that this is indeed the case. In fact, it is sufficient for Faltings to get a generic Borel reduction for the torsors and for this no assumption of splitness is required as remarked in \cite[Page 84 (d)]{ds1995}. This is good enough to get the line bundle of central charge $1$ on the moduli stack of torsors. In \S \ref{elaboration}, we give a self-contained argument deducing the result for the {\em tamely-split case} from the split case.  
\end{rem}

\begin{rem}\label{splitsurjectiveconseq} Let $\cE^{\tt I}$ be the universal Iwahori torsor on $X \times \cM^{\tt I}$.  For $z \in \cR$, a character $\chi \in \XX^*(\cG^{{\tt I}}_z)$ canonically gives rise to a line bundle $\cE^{\tt I}(\chi)$ on $\cM^{\tt I}$. By the exact sequence, this line bundle has central charge zero and comes from a line bundle on $G_z/B_z = G/B$, i.e. a character of $B$. The converse is also easily seen to hold.
\end{rem}
Using the definition given above, Zhu  shows that the central charge morphism is independent of $z \in X$ (see \cite[Proposition 4.1, and page 8]{zhu}). We give a proof below. The proof in { {\em loc.cit}} runs to a few pages. 

\begin{Cor} \label{ccconstantcy} The central charge \eqref{ccmodulistack} is independent of the base point $z$.
\end{Cor} 
\begin{proof} We first consider the case $\cG$ is the semisimple group scheme $\Hfr$ of Theorem \ref{faltingsO(1)} (which in particular could be tamely split). In this case, using the universal torsor $\cE \ra X \times \cM_{_X}(\Hfr) $ Faltings, in Theorem \ref{faltingsO(1)}, constructs a line bundle $$L(\cE)_{\Hfr}$$ which under $\glue^*_{t}: \Pic \left(\cM(\Hfr) \right) \rightarrow \Pic \left(\cF\ell_{v,x} \right)$ maps to the generator of $\Pic \left(\cF\ell_{v,x} \right)$ for any $x \in X$. In particular, its central charge is the same defined at any point $x \in X$. This proves the result for $\cG=\Hfr$.

Recall by our notation $\mathbf{a}_z$ denotes an alcove in whose closure ${\tt F}_z$ lies. For computing central charge \eqref{ccaffineflagvariety}, it suffices to further pull-back the line bundles from $\cF\ell_{{\tt F}_z}$ to $\cF\ell_{\mathbf{a}_z}$. Equivalently, since \eqref{fibproddiag} is a cartesian square, we may pull-back the line bundle by $\pi_{_{\tcs}}: \cM(\cG^{\tt I}) \ra \cM_{_X}(\cG)$ \eqref{heckecorresp} and use ${\tcm}_z: \text{Pic}(\cM(\cG^{\tt I})) \ra \mathbb{Z}$ \eqref{ccmodulistack}. {\em This reduces the general parahoric case $\cG$ to $\cG^{\tt I}$}. 

We now consider the case $\cG=\cG^{\tt I}$. In this case, the sequence \eqref{picseq} is split by $h_{_{\tcs}}^*(L(\cE)_{\Hfr})$ (see \eqref{Heckemodification}) for any $z \in X$. Thus, $\tcm_z$ is surjective for all $z \in X$ and we have 
\begin{equation} \label{splitsurjective}
\Pic \left(\cM^{\tt I} \right) = \Pic \left(\cM(\Hfr) \right) \oplus \prod_{x \in \cR} \XX^* \left(\cG^{{\tt I}}_x \right).
\end{equation}
Thus, an arbitrary line bundle $L$ on $\cM_{_X}\left(\cG^{\tt I} \right)$ may be written as a tensor product of a  power, say $p$, of $h_{_{\tcs}}^*(L(\cE)_{\Hfr})$ and $\prod_{x \in \cR} \XX^*(\cG^{{\tt I}}_x)$. For any $z \in X$, the latter have central charge zero with respect to ${\tcm}_z$ \eqref{picseq}. Thus, for any $z \in X$, 
$${\tcm}_z(L)=p,$$
and $p$ is independent of $z \in X$.
\end{proof}

\subsection{Line bundles on $\cM(\cG)$}
We now take the case $A =\cR$, i.e. the markings for the atlas coincide with the ramifications. 
 The following proposition shows that {\em all line bundles on $\prod_{x \in \cR} \cF\ell_{{\tt F}_x}$ whose central charge on each factor is a fixed multiple of $\ell$ do arise as pullbacks of line bundles from $\cM(\cG)$}. The diagram 
\eqref{fibproddiag} for $A=\cR$ induces the following Picard diagram:  
\begin{equation} \label{cartesiansquare}
\xymatrix{
\Pic \left(\cM^{\tt I} \right) \ar@{^{(}->}[rr]^{\tcg^{\tt{I}*}} && \Pic \left(\prod_{x \in \cR} \cF\ell_{\mathbf{a}_x} \right)  \\
\Pic \left(\cM_{_X}(\cG) \right) \ar@{^{(}->}[rr]^{\tcg^*} \ar@{^{(}->}[u]^{\pi_{\tcs}^*} && \Pic \left(\prod_{x \in \cR} \cF\ell_{{\tt F}_x} \right)  \ar@{^{(}->}[u]^{p_{\cR}^*}
}
\end{equation}
Let us show that the Picard group of $\cM(\cG)$ is the intersection of $\Pic \left(\cM^{\tt I} \right)$ and $\Pic \left(\prod_{x \in \cR} \cF\ell_{{\tt F}_x} \right)$ inside $\Pic \left(\prod_{x \in \cR} \cF\ell_{\mathbf{a}_x} \right)$.

\begin{prop}\label{redtoIwa} (A ``descent" criterion)   The diagram \eqref{cartesiansquare} is a pull-back square. The image of $\tcg^*$ {\em identifies with the subgroup of $\Pic \left(\prod_{x \in \cR} \cF\ell_{{\tt F}_x} \right)$ consisting of line bundles having the same central charge on each factor which is a multiple of $\ell$ (see \eqref{f})}. In particular, when $\cR$ is a single point $x$, pullback defines the isomorphism 
\begin{equation} \tcg^*: \Pic \left(\cM_{_X}(\cG) \right) \ra  \Pic \left(\cF\ell_{{\tt F}_x} \right).
\end{equation}
\end{prop}
\begin{proof} (see \cite[Proposition 7.7.2]{me}  \cite[Proposition 7.10.1 and Theorem 7.11.1]{me} for the split case) Let $\cG$ with tamely split generic fibre. Let $\cR$ be an arbitrary finite set of closed points in $X$. Since the central charge is constant for $x \in X$ by Corollary \ref{ccconstantcy} and on each factor $\cF\ell_{{\tt F}_x}$ it has to be a multiple of $l_x$ (see \eqref{f}), the above condition is necessary. We now check the sufficiency. 

We first check the sufficiency for  the Iwahori moduli stack $\cM_{_X}(\cG^{\tt I})$. The key fact again is the existence of the Faltings line bundle of central charge $1$. Recall that by \eqref{Heckemodification}, we have the morphism $h_{_{\tcs}}: \cM_{_X}(\mathcal{G}^{\tt I}) \rightarrow \cM_{_X}(\Hfr)$ where $\Hfr$ is a semisimple group scheme tamely split. In this case $\ell=1$. We saw in Corollary \ref{ccconstantcy} that in this case, the sequence \eqref{picseq} is split by $h_{_{\tcs}}^*(L(\cE)_{\Hfr})$ (see \eqref{Heckemodification}) for any $z \in X$. The pullback of $h_{_{\tcs}}^*(L(\cE)_{\Hfr})$ by $\tcg^{\tt{I}*}$  is the line bundle $L_{\alpha_{_0}}$ of Proposition \ref{PicF} in each factor $\cF\ell_{{\tt F}_x}$. Since $c(L_{\alpha_{_0}})=1$, we get central charge one in each factor.  On the other hand, line bundles of central charge zero on the atlas $\prod_{x \in \cR} \cF\ell_{{\mathbf{a}}_x}$ identify with $\prod_{x \in \cR} \XX^*(\cG_{\mathbf{a},x})$. These  identify with $\prod_{x \in \cR} \XX^*(\cG_x^{\tt I})$. Whence all line bundles of central charge zero on $\prod_{x \in \cR} \cF\ell_{{\mathbf{a}}_x}$ descend to $\cM^{\tt I}$. These constitute the other factor in the direct sum \eqref{splitsurjective} which corresponds to all line bundles of central charge zero. This proves the result for the $\cM_{_X}(\cG^{\tt I})$.  When $\cR$ is a single point $x$, we also get that $\tcg^{\tt{I}*}$ is an isomorphism. This is \cite[Proposition 7.7.1]{me} when $\cG$ is constant on $X - \cR$.

 For arbitrary $\cG$ we proceed as follows. Let us consider \eqref{fibproddiag} for $A = \cR$. Let $L_{\tt F}$ be a line bundle on $\prod_{x \in \cR} \cF\ell_{{\tt F}_x}$ having the same central charge on each factor. This condition holds trivially if $\cR$ is a single point. Let $L_{\mathbf{a}}:= p_{_{A}}^*(L_{\tt F})$ denote its pull-back under $$p_{_{A}}: \prod_{x \in \cR} \cF\ell_{\mathbf{a}_x} \ra \prod_{x \in \cR} \cF\ell_{{\tt F}_x}.$$
Then by the arguments above  $L_{\mathbf{a}}$ descends to $\cM^{\tt I}$, say as $L_{\tt I}$. 

To see the  ``descent along a group quotient" of $L_{\tt F}$, we go to \cite[proof of Proposition 1.5]{bls}. Let us consider the two quotient maps: $\tcg^{\tt I}: \prod_{x \in \cR} \cF\ell_{\mathbf{a}_x} \ra \cM_{_X}(\cG^{\tt I}) $ and $\tcg: \prod_{x \in \cR} \cF\ell_{{\tt F}_x} \ra \cM_{_X}(\cG)$. The quotienting is by the same ind-group scheme $$\text{Mor}(X - \cR, \cG^{\tt I})=\text{Mor}(X - \cR, \cG).$$ The fact that $L_{\mathbf{a}}$ descends to $\cM^{\tt I}$ is equivalent to $L_{\mathbf{a}}$ getting a  $\text{Mor}(X - \cR, \cG^{\tt I})$-linearization. By the very same token, since $L_{\mathbf{a}}:= p_{_{A}}^*(L_{\tt F})$,  it follows that $L_{\tt F}$ gets a $\text{Mor}(X - \cR, \cG)$-linearization (see self-contained argument below in \eqref{selfco}). Hence, by {{\em loc.cit}} it follows that $L_{\tt F}$ descends to a line bundle on $\cM_{_X}(\cG)$.

\begin{rem}\label{selfco}  For ease of notation, let $Z = \cF\ell_{\mathbf{a}_x}$ be the affine flag variety for the Iwahori and $Y = \cF\ell_{{\tt F}_x}$ the one for the parahoric stack. We have a morphism $p:Z \to Y$ which is a fibration with fibres products of $G/B$'s. We have a group $\Gamma$ acting on $Z$ and $Y$ compatibly: i.e. we have a commutative diagram:

\[\begin{tikzcd}
	{Z \times \Gamma } & {^{^a}_{_q}} &  Z \\
	{Y \times\Gamma} &{^{^a}_{_q}} & Y
	\arrow[from=1-1, to=1-3]
	\arrow[from=1-1, to=2-1]
	\arrow[from=1-3, to=2-3]
	\arrow[ from=2-1, to=2-3]
\end{tikzcd}\]
with the projection and action maps denoted by $a$ and $q$ respectively. We are given a line bundle $L$ on $Y$ and a ``descent datum", satisfying the usual cocycle condition, i.e. an isomorphism of line bundles on $Z \times \Gamma$: 
\begin{equation}\label{byvirtueofdesc} 
\theta_{_Z}: q^*(p^*(L)) \simeq a^*(p^*(L))
\end{equation}
This is available to us because the line bundle $p^*(L)$ descends to $Z/\Gamma$.

Since $\text{Pic}(Y \times \Gamma) \subset \text{Pic}(Z \times \Gamma)$, the isomorphism \eqref{byvirtueofdesc} gives an isomorphism 
\begin{equation}\label{newone}
\alpha_{_Y}:q^*(L) \simeq a^*(L)
\end{equation}
an isomorphism of line bundles on $Y \times \Gamma$. Note that this isomorphism   need not {\em a priori} be the same as the isomorphism $\theta_{_Z}$, and so one needs to check the cocycle condition for this $\alpha_{_Y}$.   In fact, {\em we claim} that the isomorphism $\theta_{_Z}$ over $Z \times \Gamma$ also descends to an isomorphism $\theta_{_Y}$ over $Y \times \Gamma$. Notice that over $Z \times \Gamma$ we have two isomorphisms $\theta_{_Z}$ and  $p^*(\alpha_{_Y})$. Further, $\theta_{_Z}^{-1} \circ p^*(\alpha_{_Y})$ gives an automorphism of $q^*(L)$ and hence a nowhere vanishing section of $\mathcal O$ on $Z \times \Gamma$. Since $H^0(Z \times \Gamma, \cO) = H^0(Y \times \Gamma,\cO)$ ($p$ being a $G/B$-fibration), we have
\begin{equation}
\theta_{_Z}^{-1} \circ p^*(\alpha_{_Y}) \simeq p^*(\beta_{_Y})
\end{equation}
for a $\beta_{_Y} \in H^0(Y \times \Gamma,\cO)$. Thus, $\theta_{_Z} = p^*(\alpha_{_Y} \circ \beta_{_Y}^{-1})$ and hence descends to $Y \times \Gamma$, say as $\theta_{_Y} := \alpha_{_Y} \circ \beta_{_Y}^{-1}$. Hence,  $\theta_{_Y} $ also  satisfies the cocycle condition and this implies that $L$ descends to a line bundle on $Y/\Gamma$ by \cite[proof of Proposition 1.5]{bls}.\end{rem}

\begin{rem} The perspective taken in \cite{me}, is to carry out ``descent" of the line bundle along the projective morphism $\pi_{_{\tcs}}: \cM(\cG^{\tt I}) \ra \cM_{_X}(\cG)$ \eqref{heckecorresp}.\end{rem}

\end{proof} 

\begin{Cor}\label{intergralmultip} The central charge of any line bundle on $\cM(\cG)$ is a multiple of $\ell$ (see \eqref{f}). In fact, the sequence  \eqref{picseq} becomes the exact sequence:
\begin{equation} \label{picseq1}
0 \ra \prod_{z \in \cR} \XX^*(\cG_z) \ra \Pic \left(\cM(\cG) \right) \stackrel{\tcm}{\ra} \ell.\ZZ \to 0.
\end{equation} 
\end{Cor}
\begin{proof} This follows immediately by the uniformization theorem and Corollary \ref{ccconstantcy}. The surjectivity of $\tcm$ follows from \eqref{redtoIwa}.
\end{proof}

\subsection{Interpreting pullback of line bundles under $\pi_{_{\tcs}}: \cM^{\tt I} \ra \cM(\cG)$ in terms of $h_{_{\tcs}}: \cM^{\tt I} \ra \cM(\Hfr)$ (see \eqref{heckecorresp})} \label{hdescriptionpullback} Let $L$ be  a  line bundle  on $\cM(\cG)$. We want to describe $\pi_{_{\tcs}}^*(L)$  in terms of $L(\cE)_{\Hfr}$ of \eqref{faltingsandcc1}. Firstly, its central charge  is an integral multiple of $\ell$, i.e. \begin{equation} \label{isoln1} 
\tcm(\pi_{_{\tcs}}^*(L)) = n \cdot \ell,\end{equation}  (see \eqref{f}), where the integer $n$ is dependent on the morphism $\pi_{_{\tcs}}$.
The pull-back $q^*(L)$ (see \eqref{cartesiansquare}) is a box product of some line bundles $\xi_x \rightarrow \cF\ell_{{\tt F}_x}$. Thus, the central charge of all $\xi_x$ is $n \ell$. Recall for any $x \in X$, $S({\tt F}_x) \subset \mathbb{S}$ \eqref{sx} is the subset of those affine simple roots that do not vanish at the facet ${\tt F}_x$ corresponding to the parahoric group scheme obtained by restricting $\cG$ to $\mathbb{D}_x$. Each $\xi_x$ determines an integral solution 
\begin{equation} \label{isoln} {\bf e}(x) = (\cdots, n^{x}_\alpha, \cdots) \in \ZZ^{S({\tt F}_x)}, \alpha \in S({\tt F}_x)
\end{equation} to the equation 
\begin{equation}\label{isoln2} \sum_{\alpha \in S({\tt F}_x)} n^{x}_\alpha a_{\alpha^\vee} = n \cdot \ell. \end{equation} 

Conversely, since the diagram \eqref{cartesiansquare} is pull-back square, $L$ is determined by the solution tuples $\{{\bf e}(x) | x \in \cR \}$. As in \cite{me}, for $\alpha_0$, by  $\omega_{\alpha_0}$ we mean the trivial weight. Consider the line bundle $$L_{_{{\bf e}(x)}}$$ on $G_x/B_x$ corresponding to the $B_x$-character 
\begin{equation} \label{Bxchar} {\bf e}(x)=\sum_{\alpha \in S({\tt F}_x)} n^{x}_\alpha \omega_\alpha.
\end{equation} Since \eqref{isoln} and \eqref{Bxchar} determine each other uniquely, we use the same symbol ${\bf e}(x)$ for both of them. Further, we use the same notation $L_{_{{\bf e}(x)}}$ for the corresponding line bundle on $\cM_{_X}(\cG^{\tt I})$. Then we have the following isomorphism of line bundles  on $\cM^{\tt I}$ 
\begin{equation} \label{piLh}  \pi_{_{\tcs}}^*(L)= h_{_{\tcs}}^* (L(\cE)_{\Hfr})^{n\ell} \otimes L_{_{{\bf e}(x)}}^{\otimes_{x \in \cR}}
\end{equation} for $x \in \cR$. 
 For later use we recall the following obvious fact.

 \begin{rem}\label{bwbprep} Let $V$ be a vector bundle on $\cM(\cG)$. Set $\cV:= \pi^*_{_{\tcs}}(V)$. We have an  isomorphism of vector spaces
\begin{equation} \label{isocohgps}
H^i \left(\cM(\cG),V \right)=H^i \left(\cM^{\tt I},\cV \right), \quad \text{for} \quad i \geq 0.
\end{equation}
\end{rem}
This is immediate by the Leray spectral sequence for $\pi_{_{\tcs}}$ whose fibers by \eqref{heckecorresp} are products of flag schemes \eqref{fromiwahori}.

\begin{Cor} \label{SplitGamma=Verlinde} Let $k=\mathbb{C}$. For a line bundle $L$ on $\cM_X(\cG)$, global sections are given by the Verlinde formula.
\end{Cor} 
\begin{proof} Note that $G$ is split in this case. By \eqref{isocohgps} and \eqref{piLh}, we need to compute global sections of  $h_{_{\tcs}}^* (L(\cE)_{\Hfr})^{n\ell} \otimes L_{_{{\bf e}(x)}}^{\otimes_{x \in \cR}}$. For classical groups $G$ and for the group $G_2$ these are given by the Verlinde formula by \cite[Theorem 1.2.1]{ls}. Further, this proof extends to all $G$ by  the construction of the line bundle of central charge $1$ on $\cM_X(G)$ by \cite{sorger}.
\end{proof}

\section{Conformal blocks in parahoric set-up} \label{ConformalBlocksection}
{\em The field $k$ is assumed to be of characteristic $0$}.

\subsection{Conjecture 3.7 of \cite{pr}}\label{conformalstory}  
Let $\cS$ be a finite subset of points of $X$ containing $\cR$. In Proposition \ref{redtoIwa} we have described the image of $\tcg^*: \Pic(\cM_{_X}(\cG)) \ra  \Pic(\prod_{x \in \cS} \cF\ell_{{\tt F}_x})$. Let $L$ be a line bundle on $\cM(\cG)$ and say $L$ pulls back to $\prod_{x \in \cS} \cF\ell_{{\tt F}_x}$ as $\boxtimes_{x \in \cS} L_{x}$. The conjecture claims that there exists a set $\cS$ such that we have the following equality of sections:
\begin{equation*} \text{H}^0 ( \cM(\cG), L ) =
 \text{H}^0 (\prod_{x \in \cS} \cF\ell_{{\tt F}_x}, \boxtimes_{x \in \cS} L_{x})^{\text{Lie}(\cG)(X - \cS)} = \big[\bigotimes_{x \in \cS} \text{H}^0(\cF\ell_{{\tt F}_x},L_x)\big]^{\text{Lie}(\cG)(X - \cS)}.
\end{equation*}

Since $k$ is algebraically closed of characteristic  zero, we may take recourse to the equivariant approach \cite{bs}.  Let $ \cG$ be as before a parahoric Bruhat--Tits group scheme on $X$, such that the generic fibre $\cG_{k(X)}$ is an absolutely simple, semisimple, simply connected group scheme, which splits over a Galois extension. By a classical theorem of Steinberg, $\cG_{k(X)}$ is quasi-split. Recall that in Balaji-Seshadri \cite[Theorem 5.2.7, Remark 5.2.8]{bs}, it is shown that any parahoric Bruhat--Tits group scheme $\cG$  is realized as an {\em invariant direct image} from a suitable Galois cover $\psi:Y \to X$ with Galois group $\Gamma$, i.e., there exists a {\em semisimple group scheme} $\cG'$, with split generic fibre $\cG'_{k(Y)}$ isomorphic to $G \otimes_{_{k(X)}} k(Y)$ and 
\begin{equation}\label{papparappakabaap}
\cG \simeq \psi_{_*}^{^{\Gamma}}(\cG') = \text{Res}_{_{Y/X}}^{^{\Gamma}}(\cG').
\end{equation} 

\begin{rem} \label{remarkbs} In \cite[Remark 5.2.8]{bs} it is observed that more general group schemes $\cG \ra X$ of Pappas-Rapoport \cite{pr}, (for instance, allowing local gluing functions to lie in $\Aut(G)(K_x)$), and $\cG$ to be tamely-split, can also be recovered as the invariant-direct image of $\Gamma$--group schemes which need not arise from $(\Gamma,G)$-torsors.

\end{rem}
\noindent 
(See also \cite[Theorem 1]{pr2024} for this result under tameness assumptions. See also \cite{higherBT} for higher dimensional analogues of these questions in any characteristic). After this article was written, we became aware of \cite{dh} which also gives a proof of this ``conjecture". The proof has obvious similarities in terms of techniques, but as explained later in \S\ref{clarifications}, the setting here is more general. For other comparisons between this article and {\em loc.cit}, see \S\ref{clarifications} and \eqref{clarifications1}.

\begin{thm} \label{papparappa1}  The \cite[Conjecture 3.7]{pr}  holds for $\cM(\cG)$ whenever there exists a set $\cS \supset \cR$ such that $\cG(X - \cS)$ is connected. This holds when $k=\CC$. \end{thm}
\begin{proof} 
First, we check this for the Iwahori stack. Let $ \cG=\cG^{\tt I}$ denote the Iwahori group scheme.

Let $\cG'$ on $Y$ be as above, so that we have \eqref{papparappakabaap}.  Then $\cG_{k(Y)}' \simeq \cG \otimes_{k(X)} k(Y)$ is {\em split} to which the $\Gamma$ action on $\cG_{k(Y)}'$ extends. Since $\cG$ is quasi-split, let $B_{k(X)}$ be a Borel of $\cG$ over $k(X)$. Then $B_{k(X)}$ is a form of a Borel subgroup $B_{k(Y)}' \subset \cG_{k(Y)}'$. Thus, $\Gamma$ preserves the inclusion $B_{k(Y)}' \subset \cG_{k(Y)}'$ and therefore a Borel of $G$. Taking $\Gamma$--invariants of unipotent subgroup schemes of $\cG'(k(Y))$, we get those of $\cG(k(X))$. By \cite[page 67]{faltings} or \cite[page 524]{heinloth}, $\cG(k(X))$ is generated by them.

We now show for $k = \mathbb C$,  that there exists a set $\cS$ of closed points of $X$ containing $\cR$ associated to the group scheme $\cG$ such that $\cG(X - \cS)$ is {\em connected}.  Since $\cG_{k(Y)}'$ is split, simply connected and semisimple,  it extends to an open subset $U$ of $Y$ as the constant group scheme $U \times_k G$. We set $\cS:= \psi(Y - U)$. Now we are in position to apply the equivariant extension of Drinfeld's proof in \cite[Theorem 9.5]{hongkumar}, which shows that if $\Sigma^*:= \psi^{-1}(X - \cS)$, the ind-group $\Mor_{\Gamma}(\Sigma^*,G)$ is connected. 

Let us briefly indicate the argument in this proof for the sake of the reader. The simple connectedness of the twisted affine Grassmannian in the analytic setting, is reduced to the existence of the Bruhat decomposition and thereby to the existence of a cell structure. In the split case, one is able to use this aspect of topology to enable the use of the homotopy exact sequence in the Drinfeld argument.  Now $\cR \subset \cS$ and the ind-group $\cG(X - \cS)$ identifies with $\Mor_{\Gamma}(\Sigma^*,\cG_{\Sigma^*}') = \Mor_{\Gamma}(\Sigma^*,G)$  and is therefore {\em connected}.

 The flag variety $\cF \ell_y:=(L\cG_y/L^+\cG_y)$ is connected and reduced by \cite[Theorems 0.1 and 0.2]{pradv}.   Then, by \cite[Theorem 4]{heinloth} of Heinloth (see \S \ref{uniformization}) the morphism $\tcg: \prod_{x \in \cS} \cF\ell_{{\tt F}_x} \ra \cM(\cG)$ has \'etale local sections. To prove that the ind-group $\cG(X - \cS)$ is reduced we proceed as in Laszlo-Sorger \cite[proof of Proposition 5.1]{ls}.  The morphism $\tcg$ being \'etale locally a product of an \'etale neighbourhood of $\cM(\cG)$ with $\cG(X - \cS)$, it follows that the ind-scheme of sections $\cG(X - \cS)$ is {\em reduced and connected}, and, hence integral.

We have all the ingredients to appeal to the arguments  of Laszlo and Sorger in the proof in \cite[Proposition 5.1]{ls}. The only difference is that they choose a point $p \in X -\cR$ while we work with the subset $\cS$ containing $\cR$.  Since we have checked that $\cG(X - \cS)$ is integral, applying \cite[Lemma 7.2]{bl} and then \cite[Proposition 7.4]{bl}, we have the canonical isomorphisms of sections\small
\begin{equation*}  \text{H}^0 \left( \cM^{\tt I}, L \right) =
\text{H}^0 \left(\prod_{x \in \cS} \cF\ell_{{\mathbf{a}}_x}, \boxtimes_{x \in \cS} L_{x} \right)^{\cG(X - \cS)} = \text{H}^0 \left(\prod_{x \in \cS} \cF\ell_{\mathbf{a}_x}, \boxtimes_{x \in \cS} L_{x} \right)^{\text{Lie}(\cG)(X - \cS)}.
\end{equation*}\normalsize 
To complete the proof for a general $\cG$, we use \eqref{isocohgps} and the projective fibrations $\cF \ell_{\mathbf{a}_x} \ra \cF \ell_{{\tt F}_x}$ to reduce the computation to the one on the Iwahori stack $\cM^{{\tt I}}$. In other words, we replace $\cM^{\tt I}$ by $\cM(\cG)$ and $\cF \ell_{\mathbf{a}_x}$ by $\cF \ell_{{\tt F}_x}$ in the canonical isomorphisms above.
\end{proof}

\begin{rem} We note that there was no real need to go to the Iwahori first and then deduce it for all $\cG$. The proof of the Iwahori case had a well-understood template \cite{ls} and so we take this path.
\end{rem}
\begin{rem} When $\cG$ restricted to $X - \cR$ is isomorphic to $G \times_{\CC} (X - \cR)$ e.g. \cite[Theorem 5.2.7]{bs}, the above proof shows that $\cS$ reduces to $\cR$.
\end{rem}

\section{The geometry of propagation} \label{geometrypropogationold} 
{\it In this section the characteristic of $k$ is zero.}
Let $\tcm \geq 0$, and let $C(B)$ denote the dominant Weyl chamber for the fixed Borel subgroup $B \subset G$. Let
\begin{equation}
P_{\tcm} := \{\lambda \in C(B) \mid \lambda(\theta^{\vee}) \leq \tcm \}
\end{equation}
denote the set of dominant weights of {\em level $\tcm$}. On $\cF\ell_{{\mathbf{a}}}$, the $\lambda \in P_{\tcm}$ determine the line bundle 
\begin{equation} L_{\epsilon_0}^{\otimes \tcm} \otimes L_{\lambda}.
\end{equation}
Let $\mf_{_{\tt F, \mathbf{a}}}:\cF\ell_{{\mathbf{a}}} \to \cF \ell_{\tt F}$ be the canonical morphism \eqref{theflagfibres}.

\begin{defi} \label{pcf} Let $\tcm$ be a multiple of $l({\tt F})$ \eqref{localofF}. For a facet $\tt{F}$, let 
\begin{equation}\label{levelF}
P_{\tcm}^{\tt F} := \{\lambda \in P_{\tcm} \mid L_{\epsilon_0}^{\otimes \tcm} \otimes L_{\lambda} = \mf_{_{\mathbf{a},\tt F}}^{^*}(\xi), \xi \in \text{Pic}(\cF \ell_{\tt F}) \}
\end{equation} 
i.e., the subset of $P_{\tcm}$ consisting of weights $\lambda$ such that $L_{\epsilon_0}^{\otimes \tcm} \otimes L_{\lambda}$  is the pullback of a line bundle  on $\cF \ell_{\tt F}$. 
\end{defi}   

Thus,  $P_{\tcm}^{\mathbf{a}}=P_{\tcm}$ and  $P_{\tcm}^{\tt F}$ corresponds to {\bf dominant weights} of the form \eqref{Bxchar} which, replacing $n \ell$ \eqref{f} by $\tcm$, come from {\bf positive integral} solutions in \eqref{isoln}. We denote the line bundle $\xi$ on $\cF \ell_{\tt{F}}$ in \eqref{levelF} corresponding to $\lambda \in P_{\tcm}^{\tt{F}}$ by
\begin{equation} \label{Llambdac} 
L(\lambda,\tcm).
\end{equation} 
 Let ${\tt y}:=\{y_1, \cdots, y_t \}$ denote a finite set of points in $X- \cR$. For each $x \in \cR$ (resp. $y_i$) we choose the facet $\tt{F}_x$ (resp. $\tt{F}_i$) lying in the closure of $\mathbf{a}$.  We label these with dominant weights $\lambda_x  \in P_{\tcm}^{\tt{F}_x}$ for $x \in \cR$ and $\mu_i \in P_{\tcm}^{\tt{F}_i}$ for $1 \leq i \leq t$.

\begin{defi} Let $\lambda \in P_{\tcm}^{\tt{F}}$ and $\boldsymbol\lambda=(\cdots, \lambda_x,\cdots)$ for $x \in \cR$ with $\lambda_x \in P_{\tcm}^{\tt{F}_x}$. We define the ``geometric space of pre-vacua" as:
\begin{equation}\label{prevacua}
\cH^{^{\tcge}}_{\lambda} := \text{H}^0(\cF\ell_{\tt{F}}, L) \quad \quad \cH^{^{\tcge}}_{\boldsymbol\lambda} := \text{H}^0( \prod_{x \in \cR} \cF\ell_{\tt{F}_x}, \boxtimes_{x \in \cR} L_x),
\end{equation}
and similarly $\cH^{^{\tcge}}_{\boldsymbol\mu}$ for $\boldsymbol\mu=(\cdots,\mu_i,\cdots)$ with $\mu_i \in P_{\tcm}^{\tt{F}_i}$.
\end{defi}
Thus, our $\cH^{^{\tcge}}_{\lambda}$ is dual to $\cH_{\lambda}$ in Beauville's notation \cite[page 4]{beau}. For a dominant weight $\lambda$, we denote the space of sections of $L_{\lambda}$ by 
\begin{equation}
V_{\lambda}:=H^0(G/B,L_{\lambda}).
\end{equation}
In our convention these are the Borel-Weil irreducible modules for $G$.
 
In other words, for Beauville and Laszlo-Sorger \cite{ls}, $V_{\lambda}$ is a subspace of  $\cH_{\lambda}$, in our notation {\em $V_{\lambda}$ is a quotient of $\cH^{^{\tcge}}_{\lambda}$}.
 The following is the {\it dual} analogue of \cite[Proposition 2.3 and Corollary 2.4]{beau}.

 \begin{thm} \label{propofvacua} ({\it ``Propagation of vacua"})  Let $\gfr = \gfr_{_X} := \text{Lie}(\cG)$. For $1 \leq i \leq t$, the canonical maps $V_{\mu_i} \hra \cH^{^{\tcge}}_{\mu_i}$ induce the isomorphism 
 \begin{equation} \label{2.3} [\cH^{^{\tcge}}_{\boldsymbol{\lambda}} \otimes_i V_{\mu_i}]^{\gfr(X - \cR)} = [\cH^{^{\tcge}}_{\boldsymbol{\lambda}} \otimes \cH^{^{\tcge}}_{\boldsymbol{\mu}}]^{\gfr({X - \cR \cup \tt y})}.
 \end{equation}
\end{thm}
\begin{proof}   Let $\cG_{_{{\tt I},\cR}}$ be a group scheme which is Iwahori at marked points $\cR$. Let $\cG_{_{{{\tt I},\cR \cup \tt y}}}$ be the group scheme with the same transition function as $\cG_{_{{\tt I},\cR}}$ which is Iwahori also at $\tt y$ and isomorphic to $\cG_{_{{\tt I},\cR}}$ on the complement of $\tt y$. Consider the \'etale locally trivial fibration  $\pi: \cM(\cG_{_{{{\tt I},\cR \cup \tt y}}}) \ra \cM(\cG_{_{{\tt I},\cR}})$ of algebraic stacks with fibers $(G/B)^t$, with $t = |\tt y|$ and the atlas map $\tcg: \prod_{x \in \cR} \cF \ell_{\tt{F}_x} \ra \cM(\cG_{_{{\tt I},\cR}})$. By the uniformization theorem \eqref{unif1}, the morphisms $\pi$ and $\tcg$ define the following pull-back diagram:

\begin{equation}{\begin{tikzcd}
	{ \prod_{x \in \cR} \cF \ell_{\tt{F}_x} \times (G/B)^{t}} &&& {\cM(\cG_{_{{{\tt I},\cR \cup \tt y}}})} \\
	&&& {  } \\
	{\prod_{x \in \cR} \cF \ell_{\tt{F}_x}} &   && {\cM(\cG_{_{{\tt I},\cR}})}
	\arrow["{\tcg_{_1}}", from=1-1, to=1-4]
	\arrow["{pr_{_1}}", from=1-1, to=3-1]
	\arrow["{\pi}",from=1-4, to=3-4]
	\arrow["{\tcg}",from=3-1, to=3-4]
\end{tikzcd}}\end{equation}
Then $\tcg_1$  becomes an atlas for $\cM(\cG_{_{{{\tt I},\cR \cup \tt y}}})$ which also has the alternate atlas 
\begin{equation} \label{At2} \tcg_{\cR \cup \tt y}: \prod_{x \in \cR \cup y} \cF \ell_{\tt{F}_x} \ra \cM(\cG_{_{{\tt I},\cR \cup y}}).
\end{equation}

 Define $L(\lambda_x, \tcm)$ over $\cF\ell_{\tt{F}_x}$  following \eqref{Llambdac}.   Let $L_{\cR}$ be the box tensor product of $L(\lambda_x, \tcm)$ for $x \in \cR$. Let $\cL_{\cR}$ be the line bundle on  $\cM(\cG_{_{{\tt I},\cR}})$ to which $L_{\cR}$ goes down by \eqref{redtoIwa}. Similarly, working in addition with $\cF\ell_{\tt{F}_i}$, $\mu_i$ and $\tcm$ at the $y_i$'s, we can construct line bundles $L(\mu_i, \tcm)$ and get a line bundle $\cL_{\cR \cup y}$ on $\cM(\cG_{_{{{\tt I},\cR \cup \tt y}}})$.  Let $J_{\mu_i}$ be the line bundle on $G/B$ defined by the character $\mu_i$. It defines a line bundle $\cL_{\mu_i}$ of central charge zero on $\cM(\cG_{_{{{\tt I},\cR \cup \tt y}}})$.

Thus,  the global sections of $\tcg^* \cL_{\cR}$ identify with $\cH^{^{\tcge}}_{\boldsymbol{\lambda}}$ \eqref{prevacua} and those of $J_{\mu_i}$ with $V_{\mu_i}$. Further $$\tcg_1^* ( \cL_{\cR \cup \tt y} ) \simeq \tcg^*(\cL_{\cR}) \boxtimes_i J_{\mu_i} = \boxtimes_{_{x \in \cR}} L(\lambda_x, \tcm) \boxtimes_i J_{\mu_i}.$$  Hence, its global sections identify with
\begin{equation}
\cH^{^{\tcge}}_{\boldsymbol{\lambda}} \otimes_i V_{\mu_i}.
\end{equation}

On the other hand, from the second atlas \eqref{At2}, by definition we have $$\tcg^*_{\cR \cup \tt y} (\cL_{\cR \cup \tt y})= \boxtimes_{_{x \in \cR}}L(\lambda_x, \tcm) \boxtimes_{_{i \in \tt y}} L(\mu_i, \tcm)$$ and its global sections identify with $\cH^{^{\tcge}}_{\boldsymbol{\lambda}} \otimes \cH^{^{\tcge}}_{\boldsymbol{\mu}}$. Thus, by \eqref{papparappa1} via the two atlases of $\cM \left(\cG_{_{{{\tt I},\cR \cup \tt y}}} \right)$ we have 
\begin{equation}
\left[\cH^{^{\tcge}}_{\boldsymbol{\lambda}} \otimes_i V_{\mu_i} \right]^{\gfr({X - \cR})} = \Gamma \left( \cM \left(\cG_{_{{{\tt I},\cR \cup \tt y}}} \right), \cL_{\cR \cup \tt y} \right) = \left[\cH^{^{\tcge}}_{\boldsymbol{\lambda}} \otimes \cH^{^{\tcge}}_{\boldsymbol{\mu}} \right]^{\gfr({X - \cR \cup \tt y})}.
\end{equation}

\end{proof}

\subsection{Hecke transform of vacua} \label{Hecke transform}
We now do a process on the spaces of vacua which loses or inhibits its structure at a marking.
We use the notations of \S \ref{modulistack}.

Let the atlas points $A := \{x_1, \ldots, x_s\}$ and let us express the ramifications as $\cR = A \cup \tt{y}$ for a single point $\tt y$.  We now introduce {\it the Hecke modified group scheme $\cG'$ of $\cG$ along $\cG^{\tt I} \ra \cG$ at $\tt y$} defined as follows: 
\begin{equation}
 \cG' =
\begin{cases}
\cG   & \text{away from \tt y}, \\
G  & \text{at \tt y}.  
\end{cases}
\end{equation}
More precisely, it is isomorphic to $\cG$ away from $\tt y$ and at $\tt y$ we induce the {\em standard hyperspecial structure}, i.e., $\cG'_{\tt y} \simeq G$ via $\cG^{\tt I}$. We can see both these group schemes in the framework \eqref{heckeatgplevel}:  
\[\tiny\begin{tikzcd}
	& {\cG^I} \\
	\cG && {\cG'}
	\arrow[from=1-2, to=2-1]
	\arrow[from=1-2, to=2-3]
\end{tikzcd}\]
and their associated stacks in the framework of Hecke correspondences as follows:
\begin{equation}\label{hecketwistedprop} \tiny
\begin{tikzcd}
	& { \cM^{\tt I}} \\
	{\cM(\cG)} && {\cM(\cG')} 
	\arrow["\pi", from=1-2, to=2-1]
	\arrow["\pi'"', from=1-2, to=2-3]
\end{tikzcd} \end{equation}

We now label the markings with dominant weights from $P_{\tcm}$, namely $\lambda_x$ at the $x \in \cR$ and $\lambda_{\tt y}$ to the point $\tt y$. 
For $z \in \cR$, let $L(\lambda_z, \tcm)$ denote the line bundle on $\cF \ell_{{\tt F}_z}$ (see\eqref{Llambdac})  given by the character $\lambda_z$  of the closed fiber $\cG_z$. By \eqref{redtoIwa}, $\boxtimes_{z \in \cR} L(\lambda_z, \tcm)$ on $\cQ:=\prod_{z \in \cR} \cF \ell_{{\tt F}_z}$ descends to a line bundle $L(\boldsymbol\lambda)$ on $\cM(\cG)$.  Let $\gfr = \gfr_{_X} := \text{Lie}(\cG)$. Then, we have canonical isomorphisms:
\begin{eqnarray} \otimes_{_{x \in \cR}} \cH^{^{\tcge}}_{_{\lambda_x}} & = &   H^0(\cQ, \boxtimes_{_{x \in \cR}} L(\lambda_x, \tcm)) \\
\big[ \otimes_{_{x \in \cR}} \cH^{^{\tcge}}_{_{\lambda_x}} \big]^{\gfr(X - \cR)}  &  = &  H^0(\cM(\cG), L(\boldsymbol\lambda)).
\end{eqnarray}
By \eqref{redtoIwa}, let $\boxtimes_{_{z \in A}} L(\lambda_z, \tcm)$ descend to $\cM(\cG')$ as $L(\boldsymbol{\lambda_1})$.  We have the following {\it Hecke transform of vacua}.

\begin{thm}\label{enroutetoHecke transform}   Let $L(\lambda_z)$ denote the line bundle of central charge zero on $\cM(\cG)$ given by the character $\lambda_z$ of $\cG_z$, i.e. $\tcg^*(L(\lambda_z)) = L(\lambda_z, \tcm)$.  Let $$W := \pi'_{_*}\pi^*L(\lambda_{\tt y}),$$ a vector bundle on $\cM(\cG')$. We have the following canonical isomorphism:
\begin{equation} \label{Hecke transformisom}
H^0(\cM(\cG), L(\boldsymbol\lambda)) = H^0(\cM(\cG'), L(\boldsymbol\lambda_1) \otimes W)
\end{equation}
Equivalently, viewing $W$ as an irreducible representation of $G \simeq \cG'_{\tt y}$, we have:
\begin{equation}
\big[ \otimes_{z \in \cR} \cH^{^{\tcge}}_{_{\lambda_z}} \big]^{\gfr(X - \cR)} = \big[ \otimes_{x \in A} \cH^{^{\tcge}}_{_{\lambda_x}} \otimes W \big]^{\text{Lie}(\cG')(X - A)} = \big[ \otimes_{x \in A} \cH^{^{\tcge}}_{_{\lambda_x}} \otimes W \big]^{\gfr(X - A)}
\end{equation}
In particular, if  $\lambda_{\tt y}$ is $0$,  we have an identification:
\begin{equation}
\big[ \otimes_{z \in \cR} \cH^{^{\tcge}}_{_{\lambda_z}} \big]^{\gfr(X - \cR )} = \big[ \otimes_{x \in A} \cH^{^{\tcge}}_{_{\lambda_x}} \big]^{\text{Lie}(\cG')(X - A)} = \big[ \otimes_{x \in A} \cH^{^{\tcge}}_{_{\lambda_x}} \big]^{\gfr(X - A)}
\end{equation}
\end{thm}
\begin{proof}
These follow by noting that $\pi'_* \pi^* L(\boldsymbol\lambda))= L(\boldsymbol\lambda_1) \otimes W$. 
\end{proof}

 \section{Equivariant conformal blocks and parahoric vacua}
 
{\it The field $k$ is assumed to be of characteristic zero.}
 The aim of this section is to interpret the results on the space of sections of line bundles on stacks of parahoric torsors in the light of the equivariant realisation of parahoric Bruhat--Tits group schemes as {\em invariant direct images} of semisimple group schemes on a ramified Galois cover as enunciated in \cite[Theorem 5.2.7, Remark 5.2.8]{bs} (see also \cite{pr2024} for the setting in positive characteristics). We have already seen this in \eqref{conformalstory}. Let us work in that setting in the light of \eqref{papparappakabaap}.
 
We have a semisimple group scheme $\cG'$ on $Y$ which is generically split. Let $\gfr_Y := \text{Lie}(\cG')$ be the Lie algebra of $\cG'$ and $\gfr_X:= \text{Lie}(\cG)$ the Lie algebra of $\cG$ on $X$.  Then, we see that by \eqref{papparappakabaap} we have:
\begin{equation}
\psi_{_*}^{^\Gamma}(\gfr_Y) = \gfr_X.
\end{equation}
Next, with $\cS$ containing $\cR$ as in \eqref{papparappa1}, we consider the Lie algebra $\gfr_Y \left[Y - (\psi^{-1}(\cS) \right]^{\Gamma}$
of $\Gamma$--equivariant regular sections $f: Y - \psi^{-1}(\cS) \rightarrow \gfr_Y$. Thus, 
\begin{eqnarray}\label{level1}
\cG' \left(Y - \psi^{-1}(\cS) \right)^{\Gamma}  &=&  \cG(X - \cS),\\ \label{level2}
\gfr_Y \left[Y-(\psi^{-1}(\cS) \right]^{\Gamma} &=& \gfr_X(X - \cS).
\end{eqnarray}

Let $y \in \psi^{-1}(\cS)$. Let $\gfr_{_{Y,y}}$ (resp. $\gfr_{_{X,x}}$) denote the Lie algebras of the closed fibres $\text{Lie}(\cG'_y)$ (resp. $\text{Lie}(\cG_x)$).
Then, by \eqref{papparappakabaap}, we have the identification:
\begin{equation}
\gfr_{_{Y,y}}^{^{\Gamma_y}} \simeq \gfr_{_{X,x}}
\end{equation}
where $\Gamma_y$ is the stabilizer of the $\Gamma$--action at $y \in Y$. Let $\mathbb{D}_x :=\Spec({\cO}_x)$
 and $\mathbb{D}^{^\times}_x:=\Spec(K_x)$. Then, we may consider the central extension:
\begin{equation}
\hat{\gfr}_{_{Y,y}} := \gfr_Y(\psi^{-1}(\mathbb{D}^{^\times}_x))^{\Gamma} \oplus k,
\end{equation}
where $\gfr_Y(\psi^{-1}(\mathbb{D}^{^\times}_x))$ is the space of $\Gamma$--equivariant regular sections $\psi^{-1}(\mathbb{D}^{^\times}_x) \to \gfr_Y$.
We consider the parahoric subalgebra
\begin{equation}
\hat{\mathfrak{p}}_y := \gfr_Y(\psi^{-1}(\mathbb{D}_x))^{\Gamma} \oplus k.
\end{equation}
Let $\tcm$ denote a fixed central charge. For $\lambda_x \in P_{\tcm}^{\tt{F_x}}$ (see \eqref{pcf}), let $L(\lambda_x,\tcm)$ denote the line bundle on $\cF \ell_{{\tt F}_x}$ (see \eqref{Llambdac}). 

By (\ref{redtoIwa}), the line bundle $\boxtimes_{x \in \cS} L(\lambda_x, \tcm)$ on $\prod_{x \in \cS} \cF \ell_{{\tt F}_x}$ descends to $\cM(\cG)$, say as a line bundle 
\begin{equation} \label{Llambda} L(\boldsymbol\lambda), \quad \text{for} \quad \boldsymbol\lambda = (\ldots, \lambda_x, \ldots)_{x \in \cS}.
\end{equation}
 
 We firstly have the following identification:
\begin{equation}\label{forboldsymbollambda}
\otimes_{x \in \cS} \cH^{^{\tcge}}_{_{\lambda_x}}  =   H^0 \left( \prod_{x \in \cS} \cF \ell_{{\tt F}_x}, \boxtimes_{x \in \cS} L( \lambda_x, \tcm) \right).
\end{equation}

As in the classical theory of Tsuchiya-Ueno-Yamada (see \cite{beau}), we have the integrable highest weight representation $\cH( \lambda_x)$ of $\hat{\gfr}_{_{Y,y}}$ of level $\tcm$ for $\lambda_x \in P_{\tcm}^{\tt{F}_x}$.

Recall the following notion of {\em twisted or equivariant vacua}, even for nodal curves. 
\begin{defi}(cf. \cite[Equation (17), (18)]{hongkumar}) As in \cite[page 480]{tuy} for $N$-pointed stable curves, for an equivariant pointed stable curve $Y$, we define
\begin{equation} \label{hkdefntwistedconformalblocks} {\mathcal V}_{_{Y,\Gamma,\phi}}(\cR,\boldsymbol\lambda)^{^{\dagger}} := \text{Hom}_{_{\gfr_Y [Y -\psi^{-1}(\cS)]^{\Gamma}}}\left(\otimes_{x \in \cS} \cH(\lambda_x), k \right).
\end{equation}
\end{defi}
The following corollary for $\cG$ tamely-split answers the question of a Verlinde formula for the space of sections of line bundles on $\cM_X(\cG)$ \cite[\S 8.2]{pr2024}. In the case $k=\CC$ i.e. when $G$ is split, this was earlier proven in Corollary \ref{SplitGamma=Verlinde}.
\begin{Cor} \label{thmE} The space equivariant or twisted vacua has the following geometric realization
\begin{equation}\label{georeal}
{\mathcal V}_{_{Y,\Gamma,\phi}}(\cR,\boldsymbol\lambda)^{^{\dagger}} \simeq H^0(\cM(\cG), L(\boldsymbol\lambda)).
\end{equation}
\end{Cor}
\begin{proof}
As we have remarked on notations, the integrable highest weight representation $\cH( \lambda_x)$
 is precisely the dual of the geometric space of vacua $\cH^{^{\tcge}}_{\lambda_x}$ \eqref{prevacua}. After the identification \eqref{forboldsymbollambda}, the claim follows from \eqref{level2} and Theorem \eqref{papparappa1}.
\end{proof}

\subsection{Parahoric vacua} \label{parahoricvacuasetup}
Let $X$ be a stable curve with markings at $\cR$ lying outside the node together with a parahoric group scheme $\cG$. Let $\cS$ be a finite subset of closed points of $X$ containing $\cR$ as in \eqref{papparappa1}. For each $x \in \cS$, let $\cG |_{\mathbb{D}_x}$ denote the restrictions, say corresponding to the facet $\tt{F}_x$ of $\mathbf{a}$. Let  $\tilde{L}^{+} \left( \cG |_{\mathbb{D}_x} \right)$ denote the central extensions of their jet group as in \eqref{L+} which is stated for the Iwahori group scheme. We fix a $\tcm =n \ell$ \eqref{f} for some $n \in \mathbb{N}$. For each $x \in \cS$, as in \eqref{chiL} we suppose that we are given characters 
$$\chi_x: \tilde{L}^{+} \left( \cG |_{\mathbb{D}_x} \right) \ra \mathbb{G}_m,$$ 
of the same central charge $\tcm$. Recall that by \S \ref{geometrypropogationold} each $\chi_x$ uniquely determines a $\lambda_x$ lying in $P_{\tcm}^{\tt{F}_x}$ since ${\chi_x}$ equivalently determines the line bundle $L(\lambda_x,\tcm)$ on $\cF \ell_{\tt{F}_x}$.  In other words, $\chi_x$ determines the following character  \begin{equation} (\lambda_x,\tcm): \tilde{L}^{+} ( \cG_{\mathbf{a}} ) \ra \tilde{L}^{+} ( \cG_{\tt{F}_x} )= \tilde{L}^{+} \left( \cG |_{\mathbb{D}_x} \right) \stackrel{\chi_x}{\ra} \mathbb{G}_m. \end{equation} 
This data determines a line bundle on $\prod_{x \in \cS} \tt{F}_x$. When $X$ is smooth, we get the line bundle $L(\boldsymbol\lambda)$ \eqref{Llambda} on $\cM_X(\cG)$. We have the following definition of {\it parahoric vacua} when $X$ is more generally a marked stable curve.

\begin{defi} \label{parahoricvacua} By parahoric vacua on a stable curve $X$ associated to $\left(X,\cG, \cS, \tcm, \{\chi_x \}_{x \in \cS} \right)$, we mean the Verlinde space  defined in \cite{tuy} as $${\mathcal V}_{_{X}}(\cS,\boldsymbol\lambda)^{^{\dagger}}:=\big[ \otimes_{x \in \cS} Hom(\cH(\lambda_x),k) \big]^{\text{Lie}(\cG)(X - \cS)}.$$
\end{defi}

\section{Recovery of results of Hong-Kumar \cite{hongkumar}} \label{recovery}
{\it The field $k$ is assumed to be of characteristic zero.}
The aim of this section is to derive the results of \cite{hongkumar} from those of this article with a few deviations. We note that all the results in \cite{hongkumar} are expressed in terms of the spaces of twisted covacua, while the reformulations in this article and the proofs are for the dual spaces, namely the spaces of twisted vacua.

\subsubsection {\cite[Theorem A page 2192, 2193]{hongkumar}} {\em This result is the theorem  of propagation of twisted vacua}. This can be realized as a corollary in two steps. The first step is to geometrically realize the space of twisted vacua as space of sections of a suitable line bundle of a parahoric moduli stack. This is Corollary \ref{georeal}. The second step is  the result on propagation in the geometric setting, namely Theorem \ref{propofvacua}. These two together immediately give the twisted propagation.

\subsubsection {\cite[Theorem B, page 2193]{hongkumar}} {\em This theorem is the twisted or equivariant analogue of the usual factorisation principle for the space of covacua}. We assume that the group schemes in this subsection are {\em generically split and not just tamely split} since our approach is a reduction to the work of Tsuchiya-Ueno-Yamada \cite{tuy}, which proves this in the split  setting.

The primary point behind the philosophy of ``factorisation" could be summarised as follows. It provides the framework for the computation of the dimensions of the space of vacua ${\mathcal V}_{_{X}}(\cR,\boldsymbol\lambda)^{^{\dagger}}$ on smooth projective curves $X$ of any genus $g$. The process or strategy is primarily via degenerations, thus is an inductive one. In the foundational article \cite{tuy} this is carried out by firstly connecting the marked curve $(X, \cR)$ to a marked irreducible nodal curve with a single node $(X_{_o}, \cR_{_o})$ (by constructing a family $(\mathcal X_{_S}, \cR_{_S})$) of marked curves), so that the markings continue to live on $X_{_o}$ and avoid the node $x_{_o} \in X_{_o}$. Then an analogous space of vacua  ${\mathcal V}_{_{X_{_o}}}(\cR,\boldsymbol\lambda)^{^{\dagger}}$ is defined on the curve $X_{_o}$. It is then shown that there is a canonical isomorphism $${\mathcal V}_{_{X}}(\cR,\boldsymbol\lambda)^{^{\dagger}} \simeq {\mathcal V}_{_{X_{_o}}}(\cR,\boldsymbol\lambda)^{^{\dagger}}$$
by constructing a locally free sheaf $\mathbb V$ over the family $\mathcal X_{_S}$.

The next step is to relate the space of vacua ${\mathcal V}_{_{X_{_o}}}(\cR,\boldsymbol\lambda)^{^{\dagger}}$ on the nodal curve $X_{_o}$ to the spaces of vacua on the normalization $\tilde{X}_{_o}$ (of genus $g -1$) with two extra markings $p', p'' \in \tilde{X}_{_o}$ above the node $x_{_o}$, as a decomposition of spaces of vacua indexed by the space $P_{_\tcm}$, of dominant weights of level a fixed multiple of $\ell$ (see \cite[Proposition 2.2.6]{tuy}). This is the key induction which finally ends in an explicit computation on $\mathbb P^1$.

In \cite{hongkumar}, this process is carried out step by step for the case of twisted covacua, where suitable conditions need to be imposed for the proof to work. This is the burden of work in \cite[Theorem B]{hongkumar}. 

We have the following {\em factorisation} theorem for the spaces of vacua.  Because of the splitness assumption,  we have one less case than the ones considered in \cite{hongkumar}, namely those which come from representation of the Galois group into the group of outer automorphisms of $G$. On the other hand, we do have our results even in the split case in far more generality as we have mentioned in \S12. In fact, if we assume \cite{tuy} for the Iwahori stack in the {\em tamely split case}, then we can deduce all the cases verbatim. The theorem below carries this out in the split case and the diligent reader will have no difficulty extending these to the tamely split case. Recall that $P_{_{\tcm}}$ in \cite{tuy} is $P_{_{\tcm}}^{{\mathbf{a}}}$ in our notation.

\begin{thm} \label{superparahoricfactorisation} Let $\tcm= m \cdot \ell$ with $\ell$ as in \eqref{f}, and any positive integer $m$. Let $\cG \# \cG^{^{I}}_{_{p',p''}}$ be a group scheme on ${\tilde{X}_{_o}}$ obtained by gluing the parahoric group scheme as $\cG$ on $\mathbb{D}_x$ for $x \in \cR$ with the  Iwahori group scheme at the points $p'$ and $p''$.  For $\lambda_x \in P_{\tcm}^{\tt{F}_x}$, with notations as in \eqref{Llambda} and \eqref {hkdefntwistedconformalblocks}, for a smooth projective curve $X$ we have the following decomposition theorem:
\begin{equation}\label{parahoricfactorisation}
H^0 \left(\cM_{_X} \left(\cG \right), L(\boldsymbol\lambda) \right) \simeq \bigoplus_{_{\mu \in P_{_{\tcm}}^{\mathbf{a}}}} H^0 \left(\cM_{_{\tilde{X}_{_o}}} \left(\cG \# \cG^{^{I}}_{_{p',p''}} \right), L \left(\boldsymbol\lambda, \mu,\mu' \right) \right)
\end{equation}
or equivalently with notation as in Definition \ref{parahoricvacua} we have
\begin{equation} \label{verlindefactoristion} {\mathcal V}_{_{X}} \left(\cR,\boldsymbol\lambda \right)^{^{\dagger}} \simeq  \bigoplus_{_{\mu \in P_{_{\tcm}}^{\mathbf{a}}}} {\mathcal V}_{_{\tilde{X}_{_{o}}}} \left(\cR,\boldsymbol\lambda, \mu,\mu' \right)^{^{\dagger}}.
\end{equation}
When $X$ is only a {\em marked stable curve}, \eqref{verlindefactoristion} continues to hold if we interpret the term on the left side as parahoric vacua by Definition \ref{parahoricvacua}. 
\end{thm}
\begin{proof} The proof is simply a reformulation of the result in \cite{tuy}  and \cite{ls}, together with an application of Remark \ref{bwbprep}, obtained by setting $\cG':= \cG^{\tt I}$ in \S \ref{Hecke transform} when the curve ($X$ and/or $\tilde{X}_{_o}$) is smooth. 

By \cite[Section 6]{tuy} we can restate \cite[Proposition 2.2.6]{tuy} as an isomorphism
\begin{equation}\label{tuy}
{\mathcal V}^{^{par}}_{_{X}}(\cR,\boldsymbol\eta )^{^{\dagger}} \simeq  \bigoplus_{_{\mu \in P_{_{\tcm}}}} {\mathcal V}^{^{par}}_{_{\tilde{X}_{_{o}}}}(\cR,\boldsymbol\eta, \mu,\mu')^{^{\dagger}}
\end{equation}
as a relation between the objects on smooth projective curves $X$ and $\tilde{X}_{_o}$ with markings and parabolic structures, which could be arbitrary. The data of  
$\boldsymbol\eta$ is the tuple $\{\eta_{_x} \in P_{_{\tcm}}\}_{_{x \in \cR}}$. Now using the formalism of Beauville as done in Laszlo-Sorger \cite[page 522]{ls}, the spaces on either side can be interpreted as spaces of sections of line bundles on the moduli stack of $G$-bundles with full-flag structures at the markings $\cR$, which is precisely the Iwahori stack $\cM_{_X}(\cG^{\tt I})$. This is precisely \cite[Theorem 1.2]{ls}.

More precisely, we have the following identification from \cite[Theorem 1.2]{ls}.
\begin{equation} 
{\mathcal V}^{^{par, full}}_{_{X}}(\cR,\boldsymbol\eta)^{^{\dagger}} \simeq H^0 \left(\cM_{_X} \left(\cG^{\tt I} \right), L(\boldsymbol\eta) \right)
\end{equation}
for the line bundle $L(\boldsymbol\eta)$ on $\cM_{_X}(\cG^{\tt I})$, as defined before \eqref{forboldsymbollambda}, and similarly on the normalization $\tilde{X}_{_o}$. Thus, we see that \eqref{tuy} translates for the Iwahori stack as:
\begin{equation} \label{Iwahoricfactorisation}
H^0 \left(\cM_{_X} \left(\cG^{\tt I} \right), L \left(\boldsymbol\eta \right) \right) \simeq \bigoplus_{_{\mu \in P_{_{\tcm}}^{\mathbf{a}}}} H^0 \left(\cM_{_{\tilde{X}_{_o}}} \left(\cG_{_{\tilde{X}_{_o}}}^{\tt I} \right), L(\boldsymbol\eta, \mu,\mu') \right).
\end{equation}
where we hasten to add that the right hand side of this factorisation equation, the group scheme $\cG_{_{\tilde{X}_{_o}}}^{\tt I}$ now over the normalization $\tilde{X}_{_o}$ has added Iwahori structures at the two points $p', p''$ as well.
 
When $X$ and $\tilde{X}_{_o}$  are smooth projective curves, the space of sections on either side of \eqref{Iwahoricfactorisation} can be expressed as follows. When over the smooth curve $X$, the space  $H^0 \left(\cM_{_X}(\cG), L(\boldsymbol\lambda) \right)$ can be computed by pulling back the line bundle $L(\boldsymbol\lambda)$ to the Iwahori stack by $\pi_{_{\tcs}}: \cM_{_X} \left(\cG^{\tt I} \right) \ra \cM_{_X}(\cG)$ \eqref{heckecorresp}.  

When over  $\tilde{X}_{_o}$ with extra Iwahori datum at $p', p''$, by choosing the same gluing data at the markings, we get the canonical morphism $\cG_{_{\tilde{X}_{_o}}}^{\tt I} \to \cG \# \cG^{^{I}}_{_{p',p''}}$ which is an isomorphism over the points $p',p''$ and coincides with the earlier morphism $\cG^{\tt I} \to \cG$ at the markings in $\cR$. This induces an analogous Hecke morphism for stacks of torsors on $\tilde{X}_{_o}$: $$\pi_{_{\tcs}}: \cM_{_{\tilde{X}_{_o}}}(\cG^{\tt I}) \ra \cM_{_{\tilde{X}_{_o}}}(\cG \# \cG^{^{I}}_{_{p',p''}}).$$   Whence, in either case, \eqref{Iwahoricfactorisation} implies \eqref{parahoricfactorisation} by Remark \ref{bwbprep}. 

We now consider the case $X$ is only a marked stable curve. Taking $\eta_x:=\lambda_x$ for each $x \in \cS$, notice that $\big[ \otimes_{x \in \cS} \cH(\lambda_x) \big]^{\text{Lie}(\cG)(X - \cS)}$ of Definition \ref{parahoricvacua} is defined to be the left-hand side of \eqref{tuy}. Since $\lambda_x$ belong to $P_{\tcm}^{\tt{F}_x}$,  since $\tilde{X}_{_o}$ is a smooth projective curve, by the Hecke transform result \eqref{Hecke transformisom} we have the natural isomorphism
$${\mathcal V}^{^{par}}_{_{\tilde{X}_{_{o}}}}(\cR,\boldsymbol\eta, \mu,\mu')^{^{\dagger}}={\mathcal V}_{_{\tilde{X}_{_{o}}}}(\cR,\boldsymbol\lambda, \mu,\mu')^{^{\dagger}}.$$
This proves the second assertion.
\end{proof}

\begin{rem}The proof of  \cite[Theorem B, page 2193]{hongkumar}  is an equivariant generalisation of \cite{tuy}.  Our perspective is that degeneration to nodal curve is a technique for proving the factorisation in the inductive step, while  \eqref{superparahoricfactorisation} expresses the actual point of factorisation. This isomorphism  can easily be deduced from the classically known facts if we bypass nodal curves and relate the spaces over smooth projective curves of genus $g$ with ones on genus $g-1$. \end{rem}
\begin{rem} When $X$ is only  {\em an irreducible nodal  curve, or more generally a stable curve with markings}, in  \eqref{verlindefactoristion} the space of  parahoric vacua as in Definition \ref{parahoricvacua} cannot be interpreted merely in terms of the space of sections of a line bundle on the stack of $\cG$-torsors on $X$. Such a description would involve a stack which accommodates data from the singular date on the curve. When $G$ is the linear group, this data would involve torsion-free sheaves on the nodal curve. The factorisation formula \cite{tuy} describes these spaces in terms of data on the normalisation of the nodal curve and this is what figures in the equation. \end{rem}

\subsubsection {\cite[Theorem C, page 2193]{hongkumar}} {\em This is the local freeness of the spaces of twisted vacua on smooth families of curves over a smooth base as well as the existence of a projectively flat connection.} We assume that the group schemes in this subsection are {\em generically split and not just tamely split} since our approach is a reduction to the work of Faltings \cite{faltings1993}, which proves this in the split  setting.

This follows from the results in \S \ref{globalconst}, specifically Theorem \ref{projflat}, which as we have observed is a direct corollary of the theorem of Faltings \cite{faltings1993}. A {\em drawback in our article} is that we need to assume that the genus of the curves in our family is at least $3$, a bound forced on us since we derive this from a result of Faltings, where this bound is required for computing codimensions of loci in the moduli stacks.

\subsubsection {\cite[Theorem D, page 2194]{hongkumar}} {\em This result extends the previous one and shows the local freeness of the space of twisted vacua over the Harris-Mumford compactification of the Hurwitz moduli space.}

As is written in {\em loc.cit}, the authors follow the strategy in \cite{looi} and extend it to the equivariant case. The equivariant case handled in this manner namely, via Galois covers $Y \to X$ together with a global homomorphism $\Gamma \to \text{Aut}(\mathfrak g)$ is restricted. The parahoric approach uses only the local nature of the equivariant approach.  

Our approach to constructing the bundle of vacua on the Hurwitz space is by a {\em reduction}, of the equivariant case via the parahoric approach to the basic theorem \cite[Theorem 6.2.1]{tuy} of TUY. This essentially follows by using the branch morphism from the Hurwitz stack to the base DM stack of curves and pulling back the underlying space of sections. 

More precisely, following the notations in \cite{bertin-romagny}, let us fix the data of the universal curve $(\mathcal X, \cR)$ with markings over the DM stack $\ol{\mathcal {M}}_{_{g,\cR}}$, where $n = |\cR|$. Let us fix the parahoric data $\boldsymbol \theta$ which by \cite{bs} gives rise to stable curves of genus $\tilde{g}$ with markings $\tilde{\cR}$ and Galois group $\Gamma$ with local ramification indices $\xi$ at the markings $\tilde{\cR}$. With this data $(g,\tilde{g},\cR,\xi,\Gamma)$ we have the compactified Hurwitz stack together with the {\em discriminant or ramification morphism} \cite[Proposition. 6.5.2]{bertin-romagny} 
\begin{equation}
\delta:\ol{\mathcal HM}_{_{\tilde{g},\Gamma,\xi}} \to \ol{\mathcal {M}}_{_{g,\cR}}
\end{equation}
The local freeness of the vector bundle of spaces of vacua over $ \ol{\mathcal {M}}_{_{g,\cR}}$ follows immediately from the corresponding result for the Iwahori group which is the theorem  \cite[Theorem 6.2.1]{tuy} of Tsuchiya-Ueno-Yamada. This theorem shown for the local deformation spaces, implies the local freeness of the sheaf of vacua over the moduli stack $ \ol{\mathcal {M}}_{_{g,\cR}}$. The spaces of parahoric vacua, as we have observed, is essentially a sub-class of the spaces of vacua for the Iwahori group scheme and hence the local freeness of sheaf of parahoric vacua follows. Pulling back this bundle of vacua by $\delta$ gives a locally-free sheaf on $\ol{\mathcal HM}_{_{\tilde{g},\Gamma,\xi}}$ which fits the corresponding space of twisted or equivariant vacua in a locally-free sheaf.

\subsubsection {\cite[Theorem E, page 2195]{hongkumar}}  Let $\rho:\Gamma \to \text{Aut}(G)$. Let $\mathcal{G}^{^{\rho}}$ denote the {\it parahoric Bruhat--Tits group scheme} on $X$ which comes with a homomorphism $\rho: \Gamma \ra \Aut(G)$ and  is given by the global group functor:  
$$U \rightsquigarrow G(U \times_{X} Y)^{\Gamma}$$ (see \cite[Definition. 11.1]{hongkumar}). Then, \cite[Theorem E, page 2195]{hongkumar} is subsumed in Corollary \ref{thmE}, when applied to the parahoric group scheme $\cG = \cG^{^\rho}$. For the restricted nature of the group schemes like $\cG^{^\rho}$ we refer the reader to \S \ref{clarifications}, Theorem \ref{ce}, Corollary \ref{moregen1}, and Corollary \ref{moregen}.

\subsubsection{On the hypothesis $|\Gamma|$ divides central charge $c$ in \cite[Remark 12.2 b)]{hongkumar}} In {\em loc.cit}, the authors make the remark that the assumption {\em ``the order $|\Gamma|$ divides $c$ cannot be dropped"}. 

In the set-up of \cite{bs}, we have $\ell = 1$ \eqref{f} ($\ell$ is $c$ in the notation of \cite{hongkumar}) when
\begin{enumerate}
\item $G$ is of type $A_n$,
\item the moduli of $G$-bundles with parabolic structures,
\item when each facet of $\cG$ has a hyperspecial vertex. 
\end{enumerate}
 Thus, the condition $c=1$ forces $\Gamma$ to be trivial i.e. the theory of twisted conformal blocks reduces to the untwisted one and the $\mathcal{P}arbun_{\mathcal{G}}$ of \cite[\S 11,12]{hongkumar} reduces to the usual moduli of parabolic $G$-bundles. The third case above occurs exactly when the closure of each facet does not contain the origin of $\mathbf{a}$. This happens in the case of vector bundles of fixed determinant.

To cover non-trivial parahoric cases when $G=A_n$, the order of $\Gamma_y$ {\em must be} a multiple of $n+1$ by \cite[\S 10.2]{me}. Whereas in the present article, there is no constraint on $c$ (or $\ell$ in our notation), the central charge in \cite{hongkumar} must be a multiple of $n+1$. Thus the case when the  degree of the underlying vector bundle is not a multiple of the rank lies outside the ambit of \cite{hongkumar}.

The case $c=p$, for a prime $p$ forces $\Gamma$ to be trivial or the cyclic group of order $p$. In this case, all local isotropy representations get simultaneously fixed to be $\rho: \Gamma \ra G$ and are thus equal. This forces the constancy of {\em local types}, while a natural feature of parabolic structures of Mehta-Seshadri, or even in \cite{bs}, is that local types may vary with points in $\cR$. The local isotropy representations may already be different  in the classical case $X=\mathbb{P}^1$ and $|\Gamma|=2$ studied by Ramanan and Bhosle in \cite{ramanan80}, \cite{bhosle84} and \cite{bhoslecomp84}. More precisely, in these articles the isotropy representation is referred to as the $-1$ eigenspace and in general they are allowed to vary with the Weierstrass points.

The example \cite[Remark 19(4)]{heinloth} is a miscalculation (see \cite[Remark 2.1, page 16]{zhu}). It is cited in \cite[Remark 12.2 b)]{hongkumar} as justification of the above hypothesis.

\section{On affine flag varieties and their Picard groups in the tamely split case}
{\it Unless otherwise stated, we assume $k$ to be an arbitrary field.}
The aim of this section is to make some observations on the Picard groups of affine flag varieties associated to groups which are only tamely split. We elaborate on Faltings' remark (see \eqref{faltingsandcc1}). This relies on some key facts from Pappas-Rapaport \cite[Section 7, Section 10]{pradv}. 

\subsubsection{Recall of preliminaries from \cite[Section 7b]{pradv}} \label{prelim}

Let $\cG$ be a parahoric group scheme on $X$ as in \S \ref{HeckeModdiag}. Let $K$ denote a local field which in this article will be used later for $K_x$. We now recall the general discussion in \cite[Section 7b, Proposition 7.1]{pradv} (see also \cite[\S 2.1]{zhu}) on the case when $\cG_{K}$ splits over a tamely ramified extension $\tilde{K} / K$. Let $$\Gamma= \text{Gal}(\tilde{K}/K).$$ 
Let $\tau$ denote the action of $\Gamma$ on $\tilde{K}$.  Let $H$ be a split  group scheme over $\mathbb{Z}$  such that over $\tilde{K}$ we have an isomorphism $$\cG_{K} \otimes \tilde{K} \simeq H \otimes_{\mathbb{Z}} \tilde{K}.$$ 
It is possible to choose the above isomorphism such that the natural $\Gamma$--action on $\text{Res}_{K} \left( \cG_{K} \otimes \tilde{K} \right)$ transfers to $\text{Res}_{K} \left( H \otimes \tilde{K} \right)$ as $\sigma_0 \otimes \tau$ where $\sigma_0$ is an action of $\Gamma$ on the Dynkin diagram of $H$. In the following we shall fix such an isomorphism.

Let $\tilde{\cO}:=\cO_{\tilde{K}}$. Then by \cite[7.1]{pradv}  for any point $$v \in {\mathcal B}(\cG_{K},K) \subset {\mathcal B}(H,\tilde{K}),$$  denoting $(\cG_{K})_{v}$ (resp. $( H_{\tilde{K}} )_v$) the corresponding parahoric group scheme  over $\cO$ (resp. $\tilde{\cO}$), we have a natural isomorphism
\begin{equation} \left(\cG_{K} \right)_{v} \simeq \left(\left( Res_{\tilde{\cO}/\cO} \left( \left( H_{\tilde{K}} \right)_v \right)\right)^{\Gamma} \right)^o.
\end{equation}

Furthermore, the affine flag scheme $$\cF \ell_v^{\cG}:= L\cG_{K}/ \left( L^+ \left(\cG_{K} \right) \right)_{v}$$ can be identified with the connected component of the $\Gamma$--fixed point subscheme of $$\cF \ell^{H}_{v} := LH/(L^+H)_v.$$ In particular, in the tame case, we have a closed immersion of sub-ind-scheme
 \begin{equation}\label{flagincl} \cF \ell_v^{\cG} \hookrightarrow \cF \ell^{H}_{v}.
 \end{equation} 
\begin{lem}\label{papparappa}(\cite[Section 10.a.1]{pradv})
The induced map in Picard groups associated to inclusion  \eqref{flagincl} of flag varieties is an isomorphism. 
\end{lem} 
\begin{proof} The inclusion \eqref{flagincl} induces a map of the Picard groups and this has been elaborated in \cite[10.7]{pradv}. Barring one delicate case, namely when $G$ is of type $\text{A}_{{2n}}^{{(2)}}$, where there is an element of ambiguity, in every other type it is shown in {\em loc.cit} that we get an isomorphism of the Picard groups. The delicate case  was later  sorted out by Zhu (see \cite[Remark 2.1]{zhu} which refers to Proposition 4.1  {\em loc.cit}.). This can possibly be handled using Kac \cite[page 63]{kac}. Hence, the map on Picard groups is an isomorphism in all cases. \end{proof}

\subsubsection{Elaboration on Falting's remark \cite[Remark on page 67]{faltings}} \label{elaboration} Throughout this article, we have used Theorem \ref{faltingsO(1)} whose proof beyond the split case is simply a  remark in {\em loc.cit} (see \eqref{faltingsandcc1}). The aim of this paragraph is to deduce the tamely-split case in a self-contained manner from the split case, and, \cite[Remark 5.2.8]{bs} in characteristic $0$ and \cite[Theorem 1.1]{pr2024} in positive characteristics. The proof  is by constructing a suitable Galois cover of $X$ such that there is a branch point with a single ramification point lying over it. 

We now take $\Hfr$ to be an absolutely simple, semisimple simply connected tamely split group scheme on $X$ as in Theorem \ref{faltingsO(1)}. By our tamely-split hypothesis, the restriction $\Hfr|_{k(X)}$ of the group scheme to the generic point of $X$ splits over a tamely ramified extension and the same holds for $\Hfr_{K_x}$ for any closed point $x \in X$.  By Pappas-Rapoport \cite[\S 7.a]{pradv}, the degree of the extension over $K_x$ can only be $1,2$ or $3$. In {\em loc.cit}. the list of groups from \cite[\S 4]{corvallis} (or \cite[\S 4.1-4.2]{bt}) has been summarised. Let $H$ denote the split Chevalley group scheme such that over an extension $\tilde{K}_x$ of $K_x$ we have an isomorphism
\begin{equation}\label{flagmathfrakh}
\Hfr_{K_x} \otimes \tilde{K}_x \simeq H \otimes_{\mathbb{Z}} \tilde{K}_x.
\end{equation}

For simplicity, we first assume that $\text{char}(k) = 0$ and towards the end of the construction we will make remarks regarding the generalisation in positive characteristics. As in \S \ref{conformalstory}, over fields of characteristic $0$ by \cite[Theorem 5.2.7, Remark 5.2.8]{bs}, $\Hfr$  is realized as an {\em invariant direct image} from a suitable Galois cover $\psi:Y \to X$ with Galois group $\Gamma$ of a semisimple group scheme $\cH'$ such that
\begin{enumerate}
\item[a)] the generic fibre  $\cH'_{k(Y)}$ is split and isomorphic to $H \otimes_{_{k(X)}} k(Y)$,
\item[b)] restriction of $\cH'$ to the formal neighbourhood $\mathbb{D}_y$ of any closed point $y$ in $Y$ is a split group scheme; (in fact by going to a further ramified cover it can be made isomorphic to $H \otimes \Spec(\hat{\cO}_y)$).
\item[c)] $ \Hfr \simeq \psi_{_*}^{^{\Gamma}}(\cH') = \text{Res}_{_{Y/X}}^{^{\Gamma}}(\cH')$.
\end{enumerate}

Let $\cB$ denote the set of branch points of $\psi$. Now, a finite group $\Gamma$ arises as a Galois cover of $\psi: Y \ra X$ branched at $\cB$ if every prime to $p$ quotient of $\Gamma$ is a quotient of the Fuschian group associated to $(X,\cB)$. As observed above, the tamely split cases enable us to consider totally ramified covers of index $2$ and $3$ alone. Consequently, in the following, we need only $\ZZ/2$ or $\ZZ/3$ covers and if the characteristic is accordingly not equal to $2$ and $3$, this can be realised by increasing the number of branch points $\cB$ to be more than $2$.  Depending on the type of $\Hfr$ we fix such a Galois cover $\psi:Y \ra X$. This forces for $x \in \cB$, $\psi^{-1}(x)$ to be a {\it singleton} $y$.

We now fix a closed point $x \in \cB$. Let $X^*:=X - x$ and
let $Y^*:=Y - y$. Thus, as in \S \ref{conformalstory}, with $\cS$ replaced by $x$, we have
 \begin{equation} \label{LXGGamma}
 \text{Mor}(X^*,\Hfr)=\text{Mor}_{\Gamma}(Y^*,\cH').
 \end{equation}
As an aside, we remark that under assumptions on characteristics depending on the type of the group, the above setup generalises to positive characteristics by \cite[Theorem 1.1]{pr2024}. 

The aim now is to prove the following lemma. We follow the notations in \S \ref{prelim}. We begin by observing that the conclusion of the first part of the next lemma is essentially shown in \cite[10.a.1]{pradv}.

By the uniformization theorem, the stacks $\cM_Y(\cH')$ and $\cM_{_X}(\Hfr)$ are the quotient of $ \cF \ell^{\cH'}_{v}$ and $\cF \ell_{v}^{\Hfr}$ by $\text{Mor}(Y^*,\cH')$ and $\text{Mor}_{\Gamma}(Y^*,\cH')$ respectively.

\begin{lem} (see \eqref{faltingsandcc1}) The Picard group mapping  $\Pic(\cM_{_X}(\Hfr)) \hookrightarrow \Pic(\cF \ell^{\Hfr}_v)$ induced by the quotient morphism  is an isomorphism
\end{lem}
\begin{proof}

 By Theorem \ref{faltingsO(1)} for the split case, the pull-back of the ample generator $L_{\cH'}$ of $\Pic (\cM_Y(\cH'))$ to the atlas is the ample generator of $\Pic \left(\cF \ell^{\cH'}_v \right)$. These are obviously equivariant for the action of $\text{Mor}(Y^*,\cH')$.  Consider the embedding:
$$ \cF \ell_{v}^{\Hfr}  \hookrightarrow \cF \ell^{\cH'}_{v}.$$
By \eqref{papparappa} the ample generator on $\cF \ell^{\cH'}_{v}$ restricts to the ample generator on $\cF \ell_{v}^{\Hfr}$. 

Further, this restricted line bundle is clearly equivariant under $\text{Mor}_{\Gamma}(Y^*,\cH') \subset \text{Mor}(Y^*,\cH')$ and therefore descends to $\cM_{_X}(\Hfr)$. Thus,  we have constructed a line bundle $L(\cE)_{\Hfr}$ on $\cM_{_X}(\Hfr)$ such that $\tcm_{_{F}}(L(\cE)_{\Hfr}) = 1$. Whence, the inclusion $\Pic(\cM_{_X}(\Hfr)) \hookrightarrow \Pic(\cF \ell^{\Hfr}_v)$ is an isomorphism.
\end{proof}

\section{Borel-Weil-Bott }
 The aim of this section is to show analogues of the {\tt BWB} paradigm in the setting of parahoric moduli stacks. As has been the theme throughout this article, the means to do this is a {\em reduction approach}, i.e., to reduce the questions on parahoric stacks using Hecke correspondences to the reductive stack via the Iwahori stack. In this section we reduce it to the work of Teleman \cite{bwb}.
We assume in this section that all parahoric group schemes are generically split, i.e. $\cG$ restricts to $X - \cR$ as $G \times_k (X - \cR)$. 
\subsection{Bott type constructions} {\it We assume $k$ is an algebraically closed field of characteristic zero.} Recall that $A \subset X$ denotes a non-empty finite subset of atlas points and we will take a product of flag varieties corresponding to $x \in A$ as the atlas. In this section, we choose a non-empty subset $M \subset X$ which we will call {\it $\mathfrak i$-marked points, or marked points for fixing irreducible representations on them}.  We will choose irreducible representations of the fiber of $\cG$ at each $\mathfrak i$-marked point $z \in M$. In Teleman's  \cite[\S 0.]{bwb} notation, we have $\Sigma^{c} = X$, and $$\Sigma^{c} - \Sigma=A  \quad \text{and} \quad M=\{z_1,\cdots, z_m \} \subset \Sigma=X -A.$$ Let $z \in M$ denote a $\mathfrak i$-marked point. Let $\cG_z$ denote the $z$-fibre of $\cG$ and let $${\tt G}_z := \cG_{_{z,red}}$$ be the {\em reductive quotient}. Let $${\tt B}_z \subset {\tt G}_z$$ be a Borel containing the closed fibre $\cT_z$ of the maximal torus $\cT \subset \cG$. Let $C({\tt B}_z)$ denote the dominant Weyl chamber associated to $({\tt G}_z,{\tt B}_z,\cT_z)$ and $W_z$ the Weyl group. Let  $$\ell_{_z}$$ denote the {\em length function} in the Weyl group $W_z$. {\em We emphasize that the groups vary with the points $z$}.

Let $\{\lambda_z:{\tt B}_z \to \mathbb G_m \}_{z \in M}$ be a given set of {\em regular weights} and let $\{V_z\}_{_{z \in M}}$ be the corresponding set of $\tt G_z$--irreducible representations. Let 
\begin{equation} \label{ellz} \ell_{_z}(\lambda_z) := \ell_{_z}(w_z)
\end{equation} be the length of the unique Weyl group element $w_z \in W_z$ such that 
\begin{equation} 
w_z \ast \lambda_z \in C({\tt B}_z).
\end{equation}
Recall that by the classical {\tt BWB}-identification, we have
\begin{equation} \label{bwbclassical}
V_z \simeq \text{H}^{\ell_{_z}(\lambda_z)}( {\tt G}_z/{\tt B}_z, L(\lambda_z)).
\end{equation}

 The universal torsor $\cE \ra X \times \cM(\cG)$ restricted to $z \times \cM(\cG)$ gives $\cE_z$ which is a $\cG_z$-torsor. Through morphisms defining the irreducible representation $V_z$:
\begin{equation}
\cG_z \to {\tt G_z} \to \text{GL}(V_z)
\end{equation}
on the stack $\cM(\cG)$ we get vector bundles:
\begin{equation}\label{evaluationbundles}
\mathbb V_z := \cE_z(V_z).
\end{equation}
 
In this section, {\it we will work with  Iwahori structures at $\cR \cup M $ and denote the resulting moduli stack simply by $\cM^{\tt I}$} and the resulting morphism to $\cM(\cG)$ as
\begin{equation} \label{mibwb} 
\pi_{\cR \cup M}: \cM^{\tt I} \ra \cM(\cG).
\end{equation}  
Working on $\cM^{\tt I}$ and the universal Iwahori torsor $\cE^{\tt I}$, we get $\cE_z^{\tt I}$ which is a $\cG_z^{\tt I}$-torsor. Since $\cG^{\tt I}$ is the N\'eron dilatation of $\cG$ at the $\tt B_z$'s, we get the diagram:
\begin{equation} \label{comphom}
{\cG_z^{\tt I}} \stackrel{\pi_z}{\lra} {\tt B_z} \stackrel{\lambda_z}{\lra}  {{\mathbb G}_m} 
\end{equation}
and therefore on $\cM^{\tt I}$ the {\it evaluation line bundles} 
\begin{equation} \label{evaluationlinebundles}
\mathbb L_z:= \cE_z^{\tt I}(\lambda_z \circ \pi_z).
\end{equation}

\begin{lem} With $\pi_{_{\cR \cup M}}$ as in \eqref{heckecorresp} and \eqref{ellz}, we have
\begin{equation} \label{botttype1}
R^{{m}}\pi_{_{\cR \cup M,*}}(\mathbb L_z) \simeq \mathbb V_z, \text{where}~~ m = \ell_{_z}(\lambda_z).
\end{equation}
\end{lem} 
\begin{proof} We consider the Leray spectral sequence of $\pi_{_{\cR \cup M}}$ applied to the quasi-coherent sheaf $\mathbb L_z$. By the Classical {\tt BWB} theorem (see \eqref{bwbclassical}), the spectral sequence degenerates and we get the desired isomorphism. 
\end{proof}
On the affine Grassmannians $\cF l_{\tt \mathbf{a}_z}$, via $L \cG^{\tt I} \stackrel{q}\ra \cF l_{\tt \mathbf{a}_z} = L \cG^{\tt I}/L^+({\cG^{\tt I}}|_{{\mathbb{D}_z}})$ and
\begin{equation} 
L^+({\cG^{\tt I}}|_{\mathbb{D}_z}) \stackrel{\pi_z}{\lra} {\tt B_z} \stackrel{\lambda_z}{\lra}  {{\mathbb G}_m} 
\end{equation}
  we get the line bundle associated to $q$ as (see \eqref{ucentralextension}):
\begin{equation}
	\cL_z := \tilde{L} \cG^{\tt I}(\lambda_z \circ \pi_z).
\end{equation}
 By an abuse of notation, we denote the restriction of $\cL_z$ to the fibers of $\cF l_{\mathbf{a}_z} \ra \cF l_{\tt{F}_z}$ as a line bundle on the flag variety ${\tt G}_z/{\tt B}_z$ as well:
\begin{equation} 
\cL_z \ra {\tt G}_z/{\tt B}_z.
\end{equation}

\begin{lem} With $A= \{z \}$, recall the uniformization map $\tcg: \cF\ell_{{\tt F}_z} \ra \cM(\cG)$ \eqref{unif1} for the stack $\cM(\cG^{\tt I})$. We have a canonical isomorphism:
\begin{equation}\label{fromIwahoristacktoaffinegr}
\tcg^*(\mathbb L_z) \simeq \cL_z.
\end{equation}
\end{lem}
\begin{proof} Since $\cL_z$ come from characters of ${\tt B}_z$, they are of central charge zero and hence descend to $\cM^{\tt I}$. By construction \eqref{evaluationlinebundles}, we see that the descent of $\cL_z$ equals $\mathbb{L}_z$. Thus the pullback of $\mathbb{L}_z$ by the glue morphism is $\cL_z$.
\end{proof} 

\subsection{{\tt BWB} type statements}
{\it As in \cite{bwb}, we assume $k=\CC$, thus $G$ is split and $$M \subset X -A.$$} 

We denote the pull-back of the Faltings-Sorger line bundle $L_{\Hfr}:=L(\cE)_{\Hfr}$ on $\cM(\Hfr)$ of central charge one, to the Iwahori stack $\cM^{\tt I}$ by
\begin{equation}\label{faltonIwa}
{\mathfrak I}_{\Hfr}:= h_{_{\tcs}}^*(L_{\Hfr}).
\end{equation}
Recall that for $\cM(\cG)$ we use the atlas $\prod_{x \in A} \cF \ell_{\tt{F}_x}$ i.e. we take the quotient by $\cG(X - A)$ (see \eqref{unif1}). Correspondingly, by the glue map $\tcg^{\tt I}$ of \eqref{fibproddiag}, for $\cM^{\tt I}$ \eqref{mibwb} which has Iwahori sturctures at $\cR \cup M$, the following morphism will be useful as an atlas for {\tt BWB} type statements:
\begin{equation} \label{usefulgI} 
 \tcg^{\tt I} \times \Id: \prod_{x \in A \cap \cR} \cF \ell_{\mathbf{a}_x} \times \prod_{x \in A - \cR} \cF\ell_{{\tt F}_x} \times \prod_{z \in M } {\tt G}_z/{\tt B}_z \ra \cM^{\tt I}.
\end{equation}
We will abbreviate
\begin{equation}\label{aI}
\cF \ell_{_A}^{\tt I}: = \prod_{x \in A \cap \cR} \cF \ell_{\mathbf{a}_x} \times \prod_{x \in A - \cR} \cF\ell_{{\tt F}_x} \times \prod_{z \in M } {\tt G}_z/{\tt B}_z. 
\end{equation}

Let ${\mathfrak I}_{\cF} := \tcg_{_{\tt I}}^*({\mathfrak I}_{\Hfr})$ be the pull-back of ${\mathfrak I}_{\Hfr}$ to the above flag variety for $\cM^{\tt I}$. {\it In all the statements and their proofs of this section, it is enough to assume that $A$ is a singleton}.  We have the following {\tt BWB} theorem for $\cM^{\tt I}$.

\begin{thm} ({\tt BWB} for $\cM^{\tt I}$) \label{bwbforMI} Let ${\tt n} \geq 0$. For the projection $ \cM^{\tt I} \ra B\cG (X - A)$, we have the natural Leray spectral sequence  
\begin{equation} \label{ss}
\text{H}^p \left( \cG(X - A), H^q \left( \cF \ell_{_A}^{\tt I},  \cL_z^{\boxtimes_{z \in M}} \boxtimes {\mathfrak I}^{\tt n}_{\cF} \right) \right) \implies  \text{H}^{p+q} \left(\cM^{\tt I},   \mathbb L_z^{\otimes_{z \in M}} \otimes {\mathfrak I}^{\tt n}_{\Hfr} \right).
\end{equation}

Using \eqref{ellz}, we set 
\begin{equation} \label{b1} \tcb_{_{\pi}} : = \sum_{z \in M} \ell_{_z}(\lambda_z). \end{equation} 

If $j \neq \tcb_{_{\pi}}$, we have 
\begin{equation} \text{H}^j \left(\cF \ell_{_A}^{\tt I},   \cL_z^{\boxtimes_{z \in M }} \boxtimes {\mathfrak I}^{\tt n}_{\cF} \right) = 0.
\end{equation} In particular, the spectral sequence degenerates and we have:
\small
\begin{equation*}
 \text{H} ^{j} \left(\cM^{\tt I},  \mathbb L_z^{\otimes_{z \in M}} \otimes {\mathfrak I}^{\tt n}_{\Hfr} \right) =
\begin{cases}
0  &  \quad\text{if } j \neq \tcb_{_{\pi}}, \\
\text{H}^{\tcb_{_{\pi}}} \left(\cF \ell_{_A}^{\tt I},  \cL_z^{ \boxtimes_{z \in M}} \boxtimes {\mathfrak I}^{\tt n}_{\cF} \right)^{^{\cG(X - A)}}  & \quad\text{if } j = \tcb_{_{\pi}}.
\end{cases}
\end{equation*}

\end{thm} 
\begin{proof} 

Consider the projection 
\begin{equation} \label{par} p_{_{A -\cR}}: \prod_{x \in A} \cF \ell_{\mathbf{a}_x} \times \prod_{z \in M } {\tt G}_z/{\tt B}_z \ra \prod_{x \in A \cap \cR} \cF \ell_{\mathbf{a}_x} \times \prod_{x \in A - \cR} \cF\ell_{{\tt F}_x} \times \prod_{z \in M } {\tt G}_z/{\tt B}_z.
\end{equation} To compute the cohomology of line bundles, we will pull them back and we shall denote the composite morphism as 
\begin{equation} \label{gbwb} \tcg_{_{\tt I}} : =  (\tcg^{\tt I} \times \Id) \circ p_{_{A -\cR}} : \prod_{x \in A} \cF \ell_{\mathbf{a}_x} \times \prod_{z \in M } {\tt G}_z/{\tt B}_z \ra \cM^{\tt I}.
\end{equation} 

As a preparation, let us first prove that only one cohomology, namely $\text{H}^{j} \left( \prod_{x \in A} \cF\ell_{\mathbf{a}_x} \times \prod_{z \in M} {\tt G}_z/{\tt B}_z,  \cL_z^{\otimes_{z \in M}} \otimes {\mathfrak I}^{\tt n}_{\cF} \right)$ for $j= \tcb_{_{\pi}}$ possibly does not vanish.  

To prove this, {\em we begin with a single marking $z$}, and then proceed by an easy induction. Thus, the product $\prod_{x \in A} \cF\ell_{\mathbf{a}_x} \times \prod_{z \in M} {\tt G}_z/{\tt B}_z$ shrinks to $\prod_{x \in A} \cF\ell_{\mathbf{a}_x} \times  {\tt G}_z/{\tt B}_z$. Since the line bundle $\cL_z \otimes {\mathfrak I}^{\tt n}_{\cF}$ comes from $\cF\ell_{\mathbf{a}_z}$ we may assume that the product is $\prod_{x \in A} \cF\ell_{\mathbf{a}_x} \times \cF\ell_{\mathbf{a}_z}$. 

The line bundle ${\mathfrak I}^{\tt n}_{\cF}$ already comes from a dominant character of $\tilde{L}^+ \cG$ \eqref{L+}, and the line bundle $\cL_z$ when restricted to all the terms of the product $\prod_{x \in A} \cF\ell_{\mathbf{a}_x}$ is trivial. Hence, by the Kumar--Mathieu {\tt BWB} formalism for the affine flag varieties, all higher cohomologies for the line bundle ${\mathfrak I}^{\tt n}_{\cF} \otimes \cL_z$ is contributed only from the last factor, namely, $\cF\ell_{\mathbf{a}_z}$.
Since $\cL_z$ is assumed to be regular, we have a unique Weyl group element $\varpi_z$ in the Iwahori-Weyl group such that $\varpi_z * \lambda_z \in C({\tt B}_z)$. By the identification \eqref{fromIwahoristacktoaffinegr}, and the fact that the Weyl group of the reductive quotient of the closed fibre at $z$ of $\cG$ is a subgroup of the Iwahori-Weyl group \cite{pradv}, and by the uniqueness of this element which moves the {\em regular} character $\lambda_z$ to the dominant chamber, it follows that $w_z = \varpi_z$ and $\tcb_{_{\pi}} = \ell_{_z}(\lambda_z)$.  

In other words, the Kumar--Mathieu {\tt BWB} for the affine flag variety $\cF\ell_{\mathbf{a}_z}$ then says that the cohomology of the line bundle ${\mathfrak I}^{\tt n}_{\cF} \otimes \cL_z$ vanishes in all degrees except possibly $\ell_{_z}(\lambda_z)$.

We now proceed for {\em more than one marking}. For simplicity we assume that $A=\{x \}$ is a singleton. For each added marking $z'$, we have a ${\tt G}_{z'}/{\tt B}_{z'}$-fibration of Iwahori grassmannians:
\begin{equation}
\cF\ell_{\mathbf{a}_x} \times {\tt G}_{z'}/{\tt B}_{z'} \ra \cF\ell_{\mathbf{a}_x}.
\end{equation} 
 By the obvious Leray theorem, we get the following  Kunneth--like formula: let $p = \ell_{_z}(\lambda_z)$ and $q = \ell_{_{z'}}(\lambda_{z'})$
{\small
\begin{eqnarray*}
\text{H}^{p + q} \left( \cF\ell_{\mathbf{a}_x} \times {\tt G}_{z'}/{\tt B}_{z'}, \left({\tt L}^h_{\cF} \otimes  \cL_z \right) \boxtimes \cL_{z'} \right) \\ = \text{H}^{p} \left(\cF\ell_{\mathbf{a}_x}, {\tt L}^{\tt n}_{\cF} \otimes \cL_z \right) \otimes \text{H}^{q} \left({\tt G}_{z'}/{\tt B}_{z'}, \cL_{z'} \right)
\end{eqnarray*}}
Thus, by Kumar--Mathieu {\tt BWB} and the classical {\tt BWB} for ${\tt G}_{z'}/{\tt B}_{z'}$, the cohomology group vanishes in all degrees except possibly when $p + q = \ell_{_z}(\lambda_z) +\ell_{_{z'}}(\lambda_{z'})$.

By $\tcg^{\tt I} \times \Id$ of \eqref{usefulgI}, the descent spectral sequence for the natural morphism $$\cM^{\tt I} \ra B\cG(X - A)$$ is \eqref{ss}  (see \cite[page 10 (1.9) and page 27 (5.5)]{bwb} and for a proof see \cite[Propositions 7.3.1 and (7.3.5)]{me} where only the fact that \eqref{glue} has \'etale local sections (see \S \ref{uniformization}) is used). Since it is at level two, and $H^q$ survives in at most one degree, it degenerates. 

As in the proof of Teleman \cite[Theorem 3]{bwb}, the higher cohomology groups 
$$\text{H}^p \left( \cG(X - A), H^{\tcb_{_{\pi}}} \left( \cF \ell_{_A}^{\tt I} ,  \cL_z^{\otimes_{z \in M}} \otimes {\mathfrak I}^{\tt n}_{\cF} \right) \right)$$
vanish for $p \geq 1$ by \cite[1.10, 1.11 page 10-11]{bwb}. This proves the last assertion.

\end{proof} 

Let $L_{_\cG}$ be  a  line bundle  on the parahoric stack $\cM(\cG)$ of central charge \eqref{ccmodulistack} $\tcm(L_{_\cG})$. As in \S \ref{hdescriptionpullback}, $L_{_\cG}$ pulls back on $\cM^{{\tt I}}$ \eqref{mibwb} to 
\begin{equation} \label{elli} L_{_{\tt I}} := \mathbb{L}_z^{\otimes_{z \in M}} \boxtimes \pi_{\cR \cup M}^*(L_{_\cG}) ={\mathfrak I}_{\Hfr}^{\tcm(L_{_\cG})} \otimes  \mathbb{L}_{_{{\bf e}(x)}}^{\otimes_{x \in  \cR}} \otimes \mathbb{L}_{_{\lambda_z}}^{\otimes_{z \in M}}
\end{equation}  where with ${\mathfrak I}_{\Hfr}$ as in \eqref{faltonIwa}, the evaluation line bundles $L_{_{\lambda_z}}$ are given by the ${\tt B}_z$-character $\lambda_z$ and the $L_{_{{\bf e}(x)}}$'s are given by the ${\tt B}_x$-character $$ {\bf e}(x) = \sum_{\alpha \in S({\tt F}_x)} n^{x}_\alpha \omega_\alpha.$$ 

Let  ${\tt n}$ be an integer. Recall that for $\alpha_0$, by  $\omega_{\alpha_0}$ we mean the trivial weight. Thus we may extend ${\bf e}$ to  $z \in M \setminus \cR$, where the group scheme has semisimple fibers,  as the trivial weight zero, i.e.
$${\bf e}(z) = 0, \text{for} ~~~z \in M \setminus \cR.$$

 For $z \in M$, let $\ell_{_z}(\lambda_z)$ be  associated to the characters $\lambda_z$ defining $\mathbb V_z$ as in \eqref{ellz} which involves the length function in the Weyl group $W_z$.  
On the other hand,   for $z \in M \cup \cR$, we may alternatively view $\lambda_z+ {\tt n} \cdot{\bf e}(z)$ as a $B$-character of the Weyl group of the Borel subgroup $B$ of semisimple group $G = \Hfr_z$ (for all $z \in M \cup \cR$), which occurs in the fibres of the morphism $h_{_{\tcs}}:\cM^{^{\tt I}} \to \cM(\Hfr)$. 

By the usual {\tt BWB}, (under the assumption of regularity of the character $\lambda_z+ {\tt n}\cdot {\bf e}(z)$), there is a {\em unique element} of the Weyl group $W$ of $G$,  which moves it to the dominant chamber $C(B)$. Let us denote the length of this element as $$\ell_{_W}(\lambda_z+ {\tt n}\cdot {\bf e}(z)).$$ Associated to the projection $h_{_{\tcs}}$,  we set  
\begin{equation} \label{b2} \tcb_{_h} = \tcb(\lambda,{\tt n},{\bf e}) := \sum_{z \in M \cup \cR} \ell_{_W}(\lambda_z+ {\tt n} \cdot{\bf e}(z)).
\end{equation}   

For $z \in M \cup \cR$, we set 
\begin{equation} \mathbb{V}_z := \pi_{_{\tcs} *}(\mathbb{L}_z) \end{equation} 

\begin{Cor}\label{parahoricbwb} ({\tt BWB} for $\cM(\cG)$) Let $\tt{n} \geq 0$. Assume $\cR \cup M \subset X - A$. With the gluing morphism  $$\prod_{x \in A} \cF \ell_{{\tt F}_x} \times \prod_{z \in M} {\tt G}_z/{\tt B}_z  \stackrel{p_1}{\lra} \prod_{x \in A} \cF \ell_{{\tt F}_x} \stackrel{\tcg}{\lra} \cM(\cG),$$    if all $\{ \lambda_z| z \in M \}$ are regular in $(\tt{G}_z,\tt{B}_z)$ and all $\{ \lambda_z+ {\tt n}\cdot {\bf e}(z) | z \in M \cup \cR \}$ are regular in $(G,B)$, we then have the following:
\footnotesize
\begin{equation*}
 \text{H} ^{j} \left(\cM(\cG),   \mathbb V_z^{\otimes_{z \in M}} \otimes L_{_\cG}^{\tt n} \right) =
\begin{cases}
0   & \text{if } j \neq \tcb_{_h}, \\
\text{H}^{\tcb_{_h}} \left(\prod_{x \in A} \cF\ell_{{\tt F}_x} \times \prod_{z \in M} {\tt G}_z/{\tt B}_z,   \cV_z^{\otimes_{z \in M}} \otimes \tcg^* L_{_\cG}^{\tt n} \right)^{^{\cG(X - A)}}  & \text{if } j = \tcb_{_h}.
\end{cases}
\end{equation*}
\end{Cor}
\begin{proof} Consider $\pi_{_{\cR \cup M}}: \cM^{\tt I} \ra \cM(\cG)$ of \eqref{mibwb}. We have $$\mathbb{V}_z^{\otimes_{z \in M}} \otimes L_{_\cG}^{\tt n}= R\pi_{_{\cR \cup M} *}\left(\mathbb{L}_z^{\otimes_{z \in M}} \otimes \pi_{_{\cR \cup M}}^* L_{_\cG}^{\tt{n}} \right).$$
Consequently, owing to the shifting by $\ell_z(\lambda_z)$ of \eqref{botttype1} for each $z \in M$ we have the following isomorphism of groups for all $j \geq 0$:
$$H^j \left(\cM(\cG), \mathbb{V}_z^{\otimes_{z \in M}} \otimes L_{_\cG}^{\tt n} \right) = H^{j+\tcb_{_\pi}} \left(\cM^{\tt I}, \mathbb{L}_z^{\otimes_{z \in M}} \otimes \pi_{_{\cR \cup M}}^* L_{_\cG}^{\tt n} \right).$$
Similarly consider $$p \times \Id: \prod_{x \in A} \cF\ell_{\mathbf{a}_x} \times \prod_{z \in M} {\tt G}_z/{\tt B}_z \ra \prod_{x \in A} \cF\ell_{{\tt F}_x} \times \prod_{z \in M} {\tt G}_z/{\tt B}_z.$$
From the isomorphism $$\cV_z^{\otimes_{z \in M}} \otimes \tcg^* L_{_\cG}^{\tt n} = R(p \times \Id)_* \left(\cL_z^{\otimes_{z \in M}} \otimes  \left(\tcg \circ p_1 \circ (p \times \Id) \right)^* L_{_\cG}^{\tt n} \right),$$
we deduce an isomorphism of cohomology groups for the flag varieties as well, again with a shift of $\tcb_{_\pi}$.
With notations as above and the gluing $\tcg^{\tt I} \times \Id$  \eqref{usefulgI} of  $\cM^{\tt I}$ which has Iwahori structures at $\cR \cup M$, it follows that $$\tcg \circ p_1 \circ (p \times \Id)=  \pi_{\cR \cup M} \circ \left(\tcg^{\tt I} \times \Id \right) \circ p_{_{A -\cR}} \stackrel{\eqref{gbwb}}{=} \pi_{\cR \cup M} \circ \tcg_{_{\tt{I}}}.$$
Thus 
$$\tcg_{_{\tt{I}}}^* \left(\mathbb{L}_z^{\otimes_{z \in M}} \boxtimes \pi_{\cR \cup M}^* \left(L_{_\cG}^{\tt{n}} \right) \right)=\cL_z^{\otimes_{z \in M}} \otimes  \left(\tcg \circ p_1 \circ \left(p \times \Id \right)\right)^* L_{_\cG}^{\tt n}.$$
Therefore, since $\cR \cup M \subset X - A$, letting $\cR \cup M$ play the role of $M$, we are in the position to apply \ref{bwbforMI} . The degree evaluates to $\tcb_{_h}$ by the following: $$ \mathbb{L}_z^{\otimes_{z \in M}} \boxtimes \pi_{\cR \cup M}^* \left(L_{_\cG}^{\tt{n}} \right) \stackrel{\eqref{elli}}{=} \left({\mathfrak I}_{\Hfr}^{ \tcm \left(L_{_\cG} \right)} \otimes  \mathbb{L}_{_{{\bf e}(x)}}^{\otimes_{x \in  \cR}} \right)^{\tt{n}} \otimes \mathbb{L}_{_{\lambda_z}}^{\otimes_{z \in M}}.$$

\end{proof} 

\noindent

\subsubsection{Hecke transforms of higher cohomology} We get the following generalisation of Hecke transform of vacua \eqref{Hecke transformisom}.
 
 Let $z$ be a marked point in $M \cup \cR$.
Switching the role of $\pi_{_{\tcs}}$ by $h_{_{\tcs}}$ (see \eqref{heckecorresp}), let $B \subset G = \Hfr_z$ denote a Borel of the fiber of $\Hfr$ at $z$. As in \eqref{comphom}, the composite $${\cG_z^{\tt I}} \stackrel{\pi_z}{\lra} B \stackrel{\lambda_z + {\tt n}\cdot {\bf e}(z)}{\lra}  {{\mathbb G}_m} $$  gives as in \eqref{evaluationlinebundles} the evaluation line bundle $$\mathbb H_z:= \cE_z^{\tt I} \left( \left(\lambda_z +  {\tt n} \cdot {\bf e}(z) \right) \circ \pi_z \right) = \mathbb{L}_z \otimes L_{_{{\bf e} \left(z \right)}}^ {\tt n} $$  
For gluing functions $\tcs$ now for $x \in \cR \cup M$, consider the Hecke map 
\begin{equation} h_{_{\cR \cup M}}: \cM^{\tt I} \ra \cM(\Hfr).
\end{equation} 
For $z \in M \cup \cR$, set $$\mathbb{W}_z:= h_{_{\cR \cup M} *} (\mathbb{H}_z).$$

\begin{Cor} Assume $\cR \cup M \subset X - A$. For every $ {\tt n} \geq 0 $ and every line bundle $L_{_{\cG}}$ on $\cM(\cG)$ of central charge $\tcm(L_{_\cG})$ and for every family $\{\lambda_z\}_{z \in M}$ and $\{\lambda_z+  {\tt n} \cdot {\bf e}(z)\}$ of regular weights, we have a natural isomorphism:
\begin{equation}
\text{H} ^{*-\tcb_{_{\pi}}} \left(\cM \left(\cG \right),   \mathbb V_z^{\otimes_{z \in M}} \otimes L_{_\cG}^ {\tt n} \right)  =
\text{H}^{*- \tcb_{h}} \left( \cM \left(\Hfr \right), \mathbb{W}_z^{\otimes_{z \in M}} \otimes L_{\Hfr}^{ {\tt n}  \tcm(L_{_\cG})} \right).
\end{equation}

\end{Cor}
\begin{proof} We consider the natural morphisms $\pi_{\cR \cup M}$ (resp. $h_{_{\cR \cup M}}$). Let $L_{_{\tt I}}$ \eqref{elli} be the line bundle on $\cM^{\tt I}$. As in the proof of Corollary \ref{parahoricbwb}, the cohomology group after a shift of $\tcb_{_{\pi}}$ (resp. $\tcb_{_{h}}$) may be 
identified with $\text{H} ^{*} \left(\cM \left(\cG^{\tt I} \right),   \mathbb{L}_z^{\otimes_{z \in M}} \otimes L_{_{\tt I}}^{\tt n} \right)$.
\end{proof}
\begin{rem} In this Hecke transform is involved the varying length functions $\ell_{_z}$ associated to the Weyl groups along the fibres of $\pi_{_{\tcs}}$ coming in $\tcb_{_{\pi}}$, versus the uniform single length function $\ell_{_W}$ along $h_{_{\tcs}}$ involved in $\tcb_{h}$.
\end{rem}

\section{On a theorem of Faltings \cite{faltings1993} on  projectively flat connections} \label{globalconst}

Let $q:C \ra S$ be a  family of smooth projective curves parametrized by a smooth scheme $S$. Let us fix a set $y_S \in C(S)$ of disjoint sections and for each section we choose a facet.  Let $\cG \ra C$ be a parahoric  group scheme over $C$ with parahoric structures at a section given by the chosen facet i.e.  $\cG \ra C$ has constant parahoric type across $S$. {\em We also assume that $\cG$ is of the generically split type.}

\subsection{Hecke-modification over $S$} With notations as above,
all the constructions of \S \ref{HeckeModdiag} generalise to give the sought morphisms of group schemes over $C$ leading to a Hecke-correspondence diagram {\em associated to the choice of sections} $y_{_S}$ 
of algebraic stacks over $S$.  Further, it is easily seen that the morphisms $\pi_S$ and $h_S$ are \'etale locally trivial fibrations.

We will from now on drop the subscript $S$ for the stacks.
\subsection{Projective flatness} Let $q: \cM(\cG) \ra S$ denote the structural morphism.
\begin{thm} \label{projflat} Let $g(C) \geq 3$. Let $L$ be a line bundle on $\cM(\cG)$ and let $\mathbb V_z$ be evaluation bundles \eqref{evaluationbundles} at the markings $\cR$. Then $q_{_*}(L \otimes_{z \in \cR} \mathbb V_z)$ is a vector bundle on $S$ which has a canonical Hitchin connection which is projectively flat.

\end{thm}
\begin{proof} 
Let $\mathbb L_z$ be line bundles on $\cM^{\tt I}$ as in \eqref{evaluationlinebundles}. Let $\cM^{\tt I,o}$ be the open substack of stable torsors. Albeit not expressed in this manner, but which is nevertheless an easy consequence, the result of Faltings \cite[Page 563, (ix)]{faltings1993} is that for any line bundle $\cL$ on the Iwahori stack $\cM^{\tt I}$, if we take its restriction to the stable locus $\cM^{\tt I,o}$ at the markings $\cR$, the space of sections of the line bundle $\cL \otimes \mathbb L_z$ gets a projectively flat connection. By a codimension computation, again in Faltings \cite[Page 562]{faltings1993}, the complement of the stable locus is of codimension bigger than $2$, and this codimension computation is independent of the parabolic weights chosen for stability. For this to work we assume that $k=\CC$ and the genus $g$ of curves on the fiber is at least $3$. So  by  a Hartogs argument, the space of sections on the whole stack $\cM^{\tt I}$ gets a projectively flat connection.  Whence, so does the space of sections of its direct image under the morphism $\pi_S: \cM(\cG^{\tt I}) \to \cM(\cG)$. This gives the resulting connection on $q_{_*}(L \otimes_{z \in \cR} \mathbb V_z)$ which is projectively flat. 

\end{proof}

\subsection{The regularly stable locus}
Let $\boldsymbol{\theta}$ be any choice of weights (see \cite{bs}) on $\cM(\cG)$  which induces the same weights on each moduli stack corresponding to a curve $C_s$ on the fiber over a closed point $s \in S$. Let $\cM^{rs}(\cG,\boldsymbol{\theta})$ denote the regularly stable open substack.

\begin{Cor}\label{projflatrs} The main theorem of \cite[Page 3]{biswas-mukho-wentworth} is a special case of \eqref{projflat} for the case when the parahoric structure comes from standard parabolic structures. \end{Cor}
\begin{proof}
Since the genus is at least three, by \cite[Theorem 7.2.1]{pp}, for any point $s$ of $S$ the moduli stack $\cM_{C_s}^{rs}(\cG_s,\boldsymbol{\theta})$ has complement of codimension at least two. This implies that $\cM^{rs}(\cG,\boldsymbol{\theta})$ has complement of codimension at least two. Let $M^{rs}(\cG,\boldsymbol{\theta})$ be the associated coarse space. Then it is well known that $\cM^{rs}(\cG,\boldsymbol{\theta}) \to M^{rs}(\cG,\boldsymbol{\theta})$ is a $Z_{_G}$-gerbe (see \cite[Proposition 7.1.1]{pp}). 
It follows that \eqref{projflat} holds now for the line bundle on the moduli {\em space} of regularly stable parahoric torsors.
\end{proof}

\section{Some clarifications on equivariant constructions} \label{clarifications}
The aim of this section is twofold. The first is  to briefly outline the setting of the main constructions in \cite{bs} (see also the recent article \cite{dh23} for variations and generalizations)  and relate it some later developments which are closely related to the present article, namely conformal blocks. Note that by \cite[Remark 5.3.3]{bs} the constructions in Theorems 5.2.7 and 5.3.1 only require an existence of a suitable Galois cover which is possible even when the base curve $X$ is $\mathbb{P}^1$ with at least three marked points or when $X$ is an elliptic curve. The second aim is to show x that the group schemes and stacks considered in \cite{bs} is a larger category than the ones considered in \cite{hongkumar} (see also \cite[page 262]{dh23}), making the results discussed here more general. Unfortunately, in the literature (see for example Pappas-Rapoport \cite[see Introduction and lines 2 and 3, page 797 and Example 2.15]{pr2024}), what is quoted as the main theorem of \cite{bs} is a weaker interpretation of Theorem 5.2.7 of {\em loc.cit.} and naturally this leads to limitations in applications.

Let $\psi: Y \ra X$ be a Galois cover with Galois group $\Gamma$. A $(\Gamma,G)$-bundle $E \ra Y$ is defined as a principal $G$-bundle together with a lift of the $\Gamma$--action on $Y$. By the Stein property of the local analytic disc $\mathbb{D}_y$ at any point on $Y$, at each ramification point $y$ of $\psi$, the restriction of $E$ to the disc $\mathbb{D}_y$ is a trivial $G$-torsor whose $\Gamma$ structure is determined by the representation $\rho_y: \Gamma_y \ra G$ of the isotropy subgroups $\Gamma_y \subset \Gamma$. These {\em local representation types} can be different for different branch points and do not in general induce a representation of the Galois group $\Gamma$ into $G$ or even $\text{Aut}(G)$. Somewhat different manifestations of this phenomenon and realization of parahoric group schemes on projective curves can be seen from \cite{dh23}.

The $(\Gamma,G)$-structure on $E$ gives on the adjoint group scheme $E(G) := E \times^{G} G$ ($G$ acting on itself by inner conjugation) the structure of a  $\Gamma$--group scheme  on $Y$. More generally, one may consider $\Gamma$--group schemes $\cG'$ on $Y$ which need not come as an inner twist of $G$.

\subsection{The group schemes $\cG_{_\mathbf{v}}$ of \cite{bs} on $\mathbb{P}^{^1}_{_\mathbb{C}}$}  

Now $k=\CC$. We consider parahoric group schemes on  $\mathbb{P}^{^1}$. Let $\cR$ be a set of markings on $\mathbb{P}^{^1}$ and let $m = |\cR|$. Let $G$ be a simple, simply connected group over $\mathbb{C}$. Fix an alcove $\mathbf{a}$, and a $m$-tuple $$\mathbf{v}=(v_x \in \mathbf{a} \mid x \in \cR)$$ of {\em weights}. Let the group scheme $$\cG_\mathbf{v} \ra \mathbb{P}^{^1}$$ be obtained by gluing $G \times (\mathbb{P}^{^1} - \cR)$ with the parahoric group scheme given by $v_x$ on the disc $\mathbb{D}_x$ by some gluing functions which lie in $G(K_x)$ for all $x \in \cR$.

In the next theorem by explicit computations we compare the cases considered in \cite{dh} and \cite{hongkumar} with those considered in  \cite{bs} (this is mentioned in  \cite{dh23} but more as a general remark).
\begin{thm} \label{ce}  With notations as above, there is an $m$-tuple $\mathbf{v}$ of zero-dimensional facets of  the alcove $\mathbf{a}$ such that there does not exist a Galois cover $\psi: Y \ra \mathbb{P}^1$ with Galois group $\Gamma$ and an equivariant $\Gamma$--structure on the constant group scheme $G_{_Y}:=G \times_\mathbb{C} Y$ such that $$\psi_{_*}^{^\Gamma}(G_{_Y})=\cG_{_{\mathbf{v}}}.$$
\end{thm}

\begin{proof} We will show that the moduli stacks of torsors on $\mathbb{P}^1$ corresponding to them are not isomorphic. For simplicity, we first assume that the outer automorphism group $\text{Out}(G)$ is trivial. We begin with the following trivial remark. Suppose we are given a $\Gamma$--equivariant structure on the constant group scheme $G_{_Y}=G \times_{_{k}} Y$. Equivalently, let $E_{_0}$ be the $(\Gamma,G)$-bundle underlying $G_{_Y}$. Then, the adjoint group scheme $E_{_0}(G)$ is the {\em constant group scheme} $G_{_Y}$.

We observe that since the $G$-torsor underlying $E_{_0}$ on $Y$ is trivial, it is also {\em semistable} in the sense of Mumford-Ramanathan. Since the Harder-Narasimhan reduction is canonical, it follows that $E_{_0}$ with any $\Gamma$--structure is {\em equivariantly semistable}. One could therefore take any data of local types $\tau_j$ at the set of ramifications of $\psi:Y \to \mathbb{P}^1$ and using this data on $E_{_0}$, the group scheme $$\psi_{_*}^{^\Gamma} (G_{_Y}) \simeq \psi_{_*}^{^\Gamma} (E_{_0}(G))$$
becomes a parahoric group scheme with suitable weights at the marked points on $\mathbb{P}^1$. Whence the moduli stack  $$\cM_{\mathbb{P}^{^1}} \left(\psi_{_*}^{^\Gamma} \left(G_{_Y} \right) \right)$$ has a semistable $\psi_{_*}^{^\Gamma}(G_{_Y})$ torsor in the sense of \cite[Definition. 6.3.4]{bs}. 

When $\text{Out}(G)$ is not trivial, (semi)stability of $\psi_{_*}^{^\Gamma} (G_{_Y})$-torsors is equivalent to $\Gamma$--(semi)stability of equivariant torsors under the equivariant group scheme $G_{_Y}$. Whence the above moduli stack has a semistable torsor under the above sense.  

We now turn to the stack $\cM_{\mathbb{P}^1}(\cG_\mathbf{v})$. Recall $\mathbf{v}=(v_x \mid x \in \cR)$. Since $v_x$ are zero-dimensional facets, $\mathbf{v}$ may be viewed as a weights in the sense of  \cite[6.1.2 (2)]{bs} for torsors in $\cM_{\mathbb{P}^{^1}}(\cG_{_{\mathbf{v}}})$.
  
\noindent
{\underline {Claim}}: {\em There exist choices of $\mathbf{v}$ such that $\cM_{\mathbb{P}^{^1}}(\cG_{_{\mathbf{v}}})$ has no semistable torsor. In particular, for such $\mathbf{v}$, for any Galois cover $\psi: Y \ra \mathbb{P}^1$, there does not exist an isomorphism of $\psi_{_*}^{^\Gamma}(G_{_Y})$ with $\cG_{_{\mathbf{v}}}$}.
\\

 For simplicity of exposition, we first consider the case where $Z_G$ is not trivial. For some $m \geq 3$, let us choose $m$ elements $c_i \in Z_G$ whose product is not identity. Let $v_i$ be the hyperspecial vertex of $\mathbf{a}$ corresponding to $c_i$ and $\mathbf{v}:=(v_i)$. The conjugacy class determined by $v_i$ reduces to just $c_i$. Now, by the theory of unitary representations (see  \cite[Theorem 1.0.3 (2), Corollary 8.1.8 (2)]{bs}, \cite{me2}, \cite{stability}), the moduli stack $\cM_{\mathbb{P}^1}(\cG_\mathbf{v})$ has no semistable torsor. This holds even when $\cG_{\mathbf{v}}$ is split. Note that the group scheme $\cG_{\mathbf{v}}$ is semisimple.

More precisely, for $m = 2 + n$, we take $G = E_6$ (resp. $E_7$) and note that we have at least two hyperspecial parahorics in $G(K)$. Thus,  we have elements $c_j \in Z(G)$ of order $3$ (resp. $2$).  We also have at least two non-hyperspecial maximal parahorics given by non-central elements $g_1,g_2$ whose centralizers $C_G(g_j)$ come from the two {\em distinct groups} from Borel-de-Seibenthal list of maximal semisimple subgroups of $G$. Consider the parahoric group scheme $\cG$ on $\mathbb{P}^1$ with $m$ markings $\mathbf{v}:= \{v_1, \ldots, v_n, w_1, w_2\}$ where we choose hyperspecial parahoric structure at the $n$ points $\{v_1, \ldots, v_n\}$  and two non-hyperspecial parahoric structures at $w_1, w_2$ say corresponding to finite order elements in $G$ of different orders one of which is coprime to $|Z_G|$. Choose lifts $g_1$ and $g_2$ of the conjugacy classes $w_1$ and $w_2$ respectively in $G$. By the arguments as above, it follows that $(\Pi_{_j} c_j) \cdot g_1.g_2 \neq 1$ since the centralizers of $g_j$ are distinct. Hence the stack  $\cM_{\mathbb{P}^1}(\cG_\mathbf{v})$ has {\em no semistable torsors}. 

We now consider the general case where $Z_G$ is arbitrary and may even be trivial. Let $\Delta^{ss} \subset \mathbf{a}^{m}$ denote the set of weights $\mathbf{w} \in \mathbf{a}^{m}$ such that $\cM_{\mathbb{P}^1}(\cG_\mathbf{w})$ with weight $\mathbf{w}$ has a semistable torsor. This is called the semistability polytope because by \cite[Corollary 4.13]{mw} of Meinrenken-Woodward, one knows that $\Delta^{ss}$ is a {\it closed convex}  polytope and a subset of $\mathbf{a}^{m}$ of maximal dimension. If every $m$-tuple $\mathbf{v}$ of zero-dimensional facets in $\mathbf{a}$ belongs to $\Delta^{ss}$, by convexity $\Delta^{ss}$ would be equal to $\mathbf{a}^{m}$. However, by Teleman-Woodward \cite{tw} for $m \geq 3$, in general  $\Delta^{ss} \subsetneq \mathbf{a}^{m}$.  Consequently the  $\mathbf{v}$ lying outside $\Delta^{ss}$ give moduli stacks  $\cM_{\mathbb{P}^1}(\cG_\mathbf{v})$ having no semistable torsors. This completes the proof of the claim and the theorem.

\end{proof}

\begin{Cor} \label{moregen1} The group schemes $\cG$ considered by \cite{bs}, even split ones,  form a larger class than their description in  \cite[Introduction]{pr2024} and \cite[Introduction]{swarnavatanmay} as well as the ones in \cite[Definition. 11.1]{hongkumar} and \cite[\S 3]{hongkumar2}. When $Z(G)$ is non-trivial, there are even split semisimple group schemes of this type.
\end{Cor}

\begin{Cor} \label{moregen} Moduli stacks considered by \cite{pr}, \cite{heinloth} and \cite{bs} are a larger class than those considered by \cite{hongkumar}, \cite{hongkumar2} and \cite{swarnavatanmay}, even when the group scheme $\cG$ on $X$ is split. When $Z(G)$ is non-trivial, this holds also for  semisimple group schemes $\cG$.
\end{Cor}
\begin{proof} The group schemes in \cite{pr} and \cite{heinloth} are clearly as general as \cite{bs}. Now the claim follows from Theorem \ref{ce}.
\end{proof} 
\subsubsection{On \cite{dh}}\label{clarifications1} The article \cite{dh} appeared before we put up this article in the math arxiv. However, both the motivation as well as the objects considered have some distinctions which we wish to briefly dwell on here. 
\begin{itemize}
\item As has been shown in this section, the setting in \cite{dh} is limited by the defining group scheme (see page 2, third paragraph loc. cit.) and the corresponding stack of torsors. To be more specific, the setting in most parts of {\em loc.cit}, for instance, in \cite[Section 4]{dh},  is exactly that of \cite[Section 11, Definition. 11.1]{hongkumar} where the authors themselves say that it provides only {\em a class of examples of parahoric Bruhat--Tits group schemes}, whereas, working on the base curve, the parahoric setting we have chosen is most general. 
\item It seems that \cite[Latest archival version]{dh} are unaware of the work \cite{me} by Pandey  where several of the themes touched on have been dealt with in the split case and in the present article in the tamely split case. Here below we cite a few.
\begin{itemize}
\item   The conjecture on \cite[Page 8]{dh} is already proven in Proposition \ref{redtoIwa} above. In fact, \cite[equation (12)]{dh}  is proven in \eqref{redtoIwa}.

\item In the case $G$ is split, \cite[Theorem D, introduction]{dh} is already proven in \cite[Proposition  7.7.1 (page 28), Theorem 7.11.1 page 34]{me},  firstly for one ramification point and later for multiple ramification points.

\item The equation \cite[Equation(13) on Page 8]{dh} is precisely \cite[between Theorems 1.0.2 and 1.0.3 on page 2]{me}.

\item The proof of Pappas-Rapoport conjecture in \cite{dh} together with the conjecture on \cite[Page 8]{dh} clearly deals with a   smaller class of objects and spaces than the ones considered in \cite{bs} as elaborated in Theoerem \ref{ce}. 

\end{itemize}
\end{itemize}

\bibliographystyle{alpha}
\bibliography{Verlinde}
%\end{thebibliography}

\end{document}